\documentclass[11pt, a4paper, fleqn]{article}
\usepackage[latin1,utf8]{inputenc}
\usepackage{amsmath}
\usepackage{amsthm}
\usepackage{amsfonts}
\usepackage{amssymb}
\usepackage[T1]{fontenc}
\usepackage[all]{xy}
\usepackage{bbm}
\usepackage{enumitem}
\usepackage{color}
\usepackage[hidelinks]{hyperref}
\usepackage{graphicx}
\usepackage{transparent}
\usepackage{graphics}
\usepackage{xcolor}
\usepackage{color}
\usepackage{geometry}
\usepackage{marginnote}

\usepackage{pgf,tikz}
\usepackage{mathrsfs}
\usetikzlibrary{arrows,decorations.pathmorphing,backgrounds,positioning,fit,petri}

\title{Stability conditions, $\tau$-tilting Theory and Maximal Green Sequences}

\author{Thomas Br\"ustle, David Smith and Hipolito Treffinger}

\theoremstyle{plain} 
\newtheorem{theorem}{Theorem}[section]
\newtheorem{prop}[theorem]{Proposition}
\newtheorem{lem}[theorem]{Lemma}
\newtheorem{cor}[theorem]{Corollary}

\theoremstyle{remark}
\newtheorem{rmk}[theorem]{Remark}
\newtheorem{ex}[theorem]{Example}

\theoremstyle{definition}
\newtheorem{defi}[theorem]{Definition}
\newtheorem{conj}[theorem]{Conjecture}

\newcommand{\rep}[1]{%
  {%
    \tiny%
    \begin{matrix}%
      #1%
    \end{matrix}%
  }%
}

\def\Hom{\mbox{Hom}}

\def\coker{\mbox{coker}}
\def\im{\mbox{im}}
\def\Fac{\mbox{Fac}\,}
\def\Sub{\mbox{Sub}\,}
\def\add{\mbox{add}\,}
\def\mod{\mbox{mod}\,}
\newcommand{\A}{\mathcal{A}}
\newcommand{\X}{\mathcal{X}}
\newcommand{\F}{\mathcal{F}}
\renewcommand{\P}{\mathcal{P}}
\newcommand{\T}{\mathcal{T}}
\newcommand{\ra}{\rightarrow}

\newcommand{\D}{\mathfrak{D}}
\newcommand{\Ch}{\mathfrak{C}}
\newcommand{\Co}{\mathcal{C}}

\usepackage[nottoc]{tocbibind}

\begin{document}

\maketitle

\begin{center}
Dedicated to Idun Reiten on the occasion of her 75th birthday.
\end{center}

\abstract{Extending the notion of maximal green sequences to an abelian category, we characterize the stability functions, as defined by Rudakov, that induce a maximal green sequence in an abelian length category. Furthermore, we use $\tau$-tilting theory to give a description of the wall and chamber structure of any finite dimensional algebra. Finally we introduce the notion of green paths in the wall and chamber structure of an algebra and show that green paths serve as geometrical generalization of maximal green sequences in this context. }

\section{Introduction}
Stability conditions were introduced in representation theory of quivers in seminal papers by Schofield \cite{Scho} and King \cite{Ki}, and the general notion of stability was formalised in the context of abelian categories by Rudakov \cite{Ru}.
Since then, the study of rings of quiver semi-invariants by Derksen and Weyman has been expanded to the context of cluster algebras. The work of Igusa, Orr, Todorov and Weyman \cite{IOTW} shows that walls in the semi-invariant picture correspond to the $c$-vectors in cluster theory. These vectors are also studied in quantum field theory, where they are interpreted as charges of BPS particles. It turns out that maximal green sequences, that is, maximal paths in the semi-invariant picture which are oriented in positive direction, give rise to a complete sequence of charges, called spectrum of a BPS particle. 

The semi-invariant picture of quiver representations has re-appeared in mathematical physics and mirror symmetry as scattering diagrams such as in Kontsevich and Soibelman's study of wall crossing in the context of Donaldson-Thomas invariants in integrable systems and mirror symmetry. 
Bridgeland describes in \cite{B16} the  wall and chamber structure in this context.

It seems natural to relate these concepts to the recent work of Adachi, Iyama and Reiten on $\tau$-tilting pairs in representation theory. The aim of this paper is therefore to join the concept scattering diagrams and their wall-and-chamber structure as described in \cite{B16} with the combinatorial structure of the fan associated with $\tau$-tilting modules as given in \cite{DIJ}, as well as to investigate maximal green sequences, and their continuous counterparts in the stability space.
\bigskip

\subsection{Content}

We first study Rudakov's notion of stability on an abelian length category $\mathcal{A}$, which is given by a function  $\phi$ on Obj($\mathcal{A}$) that assigns to each object $X$  a phase $\phi(X)$ in a totally ordered set, satisfying the so-called see-saw condition on short exact sequences, see
definition \ref{seesaw}.
 An object $0\neq M$ in $\mathcal{A}$ is said to be \textit{$\phi$-stable} (or \textit{$\phi$-semi-stable}) if every nontrivial subobject $L\subset M$ satisfies $\phi(L)< \phi(M)$ ( or $\phi(L)\leq \phi(M)$, respectively).
Inspired by \cite{B16}, but in the more general context of abelian categories allowing infinitely many simple objects, we then define for each phase $p$ a torsion pair $(\T_p, \F_p)$ in $\mathcal{A}$ as follows (see proposition \ref{torsionpair}):

$$\T_p=\{M\in\A \ : \phi(N)\geq p \text{ for every quotient $N$ of $M$}\}$$
$$\F_p=\{M\in\A \ : \ \phi(N)< p \text{ for every subobject $N$ of $M$}\}$$

Since $\T_p\supseteq\T_q$ when $p\leq q$,  a stability function $\phi$ induces a chain of torsion classes in $\A$.  
We define a maximal green sequence in $\A$ to be a not refinable finite increasing chain of torsion classes starting with the zero class and ending in $\A$.
Following Engenhorst \cite{E14}, we call a stability function $\phi$ on $ \A$ \textit{discrete} if it admits (up to isomorphism) at most one stable object for every phase at $p$. 

The main result in section 3  characterizes which stability functions induce maximal green sequences in $\A$, see Theorem \ref{maximalgreensequences}:

\begin{theorem}
Let $\phi:\A\ra \P$ be a stability function that admits no maximal phase. Then $\phi$ induces a maximal green sequence of torsion classes in $\A$ if and only if $\phi$ is a discrete stability function inducing only finitely many different torsion classes $\T_p$.
\end{theorem}

In section 4 we recall  the notion of stability studied by King  \cite{Ki}:
Let $\A$ be an abelian category and $\theta:K_0(\A)\ra \mathbb{R}$ an additive function on the Grothendieck group of $\A$. Then an object $M\in\A$ is called  \textit{$\theta$-semi-stable} if $\theta(M)=0$ and $\theta(L) \leq 0$  for every  subobject $L$ of $M$.

Suppose now that $\A$ has $n$ simple objects, hence $K_0(\A)$ is isomorphic to $\mathbb{Z}^n$. 
We denote by $(\mathbb{R}^n)^*$ the dual vector space of $\mathbb{R}^n$. Then, for an object $M$  of $\A$, we define the \textit{stability space of $M$} to be 
$$\D(M)=\{\theta\in(\mathbb{R}^n)^* : M \text{ is $\theta$-semi-stable}\}.$$

The stability space $\D(M)$ of $M$ is contained in the hyperplane defined by the linear form $\theta$, but it could have smaller dimension. We say that $\D(M)$ is a \textit{wall} when $\D(M)$ has codimension one. 
Outside the walls, there are only linear functions $\theta$ having no $\theta$-semi-stable modules other that the zero object.
Removing the closure of all walls we obtain a space
$$\mathfrak{R}=(\mathbb{R}^t)^*\setminus\overline{\bigcup\limits_{M\in\A}\mathcal{D}(M)}.$$
A connected component $\mathfrak{C}$ of dimension $n$ of $\mathfrak{R}$ is called a \textit{chamber}.  

From now on, we consider the case when $\A$ is the category of finitely generated modules $\mod A$ over a finite dimensional algebra $A$ over an algebraically closed field $k$.
In this context, we study the notion of $\tau$-rigid modules and $\tau$-tilting pairs from Adachi, Iyama and Reiten \cite{AIR}.
It has been shown by King that the category of $\theta$-semi-stable modules is abelian. Using $\tau$-tilting theory we are able to give a more precise statement (see Theorem \ref{stablemodcat}): 

\begin{theorem}
Let $(M,P)$ be a $\tau$-rigid pair. If the linear map $\theta$ is induced by the $g$-vectors of $(M,P)$, then there exists an algebra $C$ such that the category of $\theta$-semi-stable modules is equivalent to $\emph{mod} C$.
\end{theorem}

We further show in section 4 how the $\tau$-tilting fan introduced by Demonet, Iyama and Jasso in \cite{DIJ} can be embedded into King's stability manifold:
Each $\tau$-tilting pair $(M,P)$ yields a chamber $\Ch_{(M,P)}$, and one can give a complete description of the walls bordering this chamber $\Ch_{(M,P)}$:

\begin{theorem}\label{introwalls}
Let $A$ be a finite-dimensional algebra over an algebraically closed field. Then there is an injective function $\Ch$ mapping the $\tau$-tilting pair $(M,P)$ onto a chamber $\mathfrak{C}_{(M,P)}$ of the wall and chamber structure of $A$. Furthermore, if $A$ is $\tau$-tilting finite then $\Ch$ is also surjective.
\end{theorem}

We also define in section 4 a function $\T$ which assigns to each chamber $\Ch$ a torsion class $\T_{\Ch}$, and  show that $\T_{\Ch_{(M,P)}}=\Fac M$.

Finally, we generalize in section 5 the notion of maximal green sequences to green paths. This idea is related to work in \cite{Mmutnoninv}, \cite{GHKK} and \cite{B16}:
Consider a continuous function $\gamma:[0,1]\rightarrow \mathbb{R}^n$  such that $\gamma(0)=(1,\dots, 1)$ and $\gamma(1)=(-1,\dots,-1)$. If we fix an $A$-module $M$, $\gamma(t)$ induces a continuous function $\rho_M:[0,1]\rightarrow \mathbb{R}$ defined as $\rho_M(t)=\theta_{\gamma(t)}([M])$, where $\theta_{y}$ denotes the linear map $x\mapsto  \langle x, y\rangle$. Note that $\rho_M(0)>0$ and $\rho_M(1)<0$. Therefore, for every module $M$ there is at least one $t\in(0,1)$ such that $\theta_{\gamma(t)}([M])=0$. This easy remark leads us to the definition of green paths as follows:

\begin{defi}
A \textit{green path in $\mathbb{R}^n$} is a continuous function $\gamma: [0,1]\ra\mathbb{R}^n$ such that $\gamma(0)=(1,\dots, 1)$, $\gamma(1)=(-1,\dots,-1)$ and such that for every $M\in$ mod $A$ there is exactly one $t_M\in[0,1]$ such that $\theta_{\gamma(t_M)}([M])=0$.
\end{defi}

Note that we allow green paths to pass through the intersection of walls. It is easy to see that every green path $\gamma:[0,1]\to \mathbb{R}^n$ induces a stability function $\phi_{\gamma}:\rm{mod} A\to [0,1]$ in the sense of Rudakov defined by $\phi_{\gamma}(M)=t_M$, where $t_M$ is the unique element in $[0,1]$ such that $\theta_{\gamma(t_M)}(M)=0$. Moreover $M$ is $\phi_{\gamma}$-semi-stable if and only if $M$ is $\theta_{\gamma(t_M)}$-semi-stable. Equipped with this notion, we are finally prepared to show that green paths generalize maximal green sequences in module categories:

\begin{theorem}
Let $\gamma$ be a green path and suppose that there are only finitely many $\phi_{\gamma}$-stable modules $M_1, \dots, M_n$. Then $\gamma$ induces a chain of torsion classes which forms a maximal green sequence if and only if $t_{M_i}\neq t_{M_j}$ for all $i\neq j$. Moreover, every maximal green sequence is obtained in this way.
\end{theorem}

\subsection{Related work}
When reporting our results at an Oberwolfach workshop, we learned that Speyer and Thomas are working on related questions. In particular, they describe the chambers given by $\tau$-tilting pairs $(M,P)$ as we did in Theorem \ref{introwalls}; for more details we refer to  their Oberwolfach report (workshop ID: 1708).

Also, when reporting our results at the Auslander conference we were informed that Igusa obtained a geometric characterization of maximal green sequences using similar arguments in the case of hereditary and Jacobian algebras, see \cite{Igu}.

Finally, would like to point out that Yurikusa is using the $g$-vectors of $2$-term silting complexes of $D^b(A)$ to describe a bijection between left wide subcategories of $\mod A$ and the left finite semi-stable subcategories of $\mod A$, induced by a linear map $\theta$ in a similar way as we do in this article, see \cite{Yur}.
\bigskip

We refer to the textbooks \cite{ARS,AsSS,bookRalf} for background material.

\section{Stability Conditions}
The aim of this section is to study Rudakov's \cite{Ru} definition of stability on abelian categories. While \cite{Ru} uses the notion of a proset, we prefer to work with stability functions. 
We first review this concept of stability here, and then discuss torsion classes arising from a stability function.

\subsection{Stability Conditions}
Throughout this section, we consider an essentially small abelian category $\mathcal{A}$.

\begin{defi}\label{seesaw}
Let $(\mathcal{P},\leq)$ be a totally ordered set and  $\phi :  Obj(\mathcal{A}) \to \mathcal{P}$ a function on $\mathcal{A}$  which is constant on isomorphism classes. For an object $x$ of $\A$, we refer to $\phi(x)$ as the \textit{phase (or slope)} of $x$.
Following \cite[Definition 1.1]{Ru}, 
the map $\phi$ is called a \textit{stability function} if for each short exact sequence  $0 \to L \to M \to N \to 0$ of nonzero objects in $\mathcal{A}$ one has the so-called \textit{see-saw (or teeter-totter) property}, see figure 1:
$$
\begin{array}{ll}
\text{either} & \phi(L) < \phi(M) < \phi(N), \\
\text{or} & \phi(L) > \phi(M) > \phi(N),\\
\text{or} & \phi(L) = \phi(M) = \phi(N).
\end{array}
$$
\end{defi}

\begin{figure}
$$
\begin{array}{ccc}
\xymatrix{ & &\phi(N) \\
           &\phi(M)\ar@{-}[d]\ar@{-}[ru] & \\
           \phi(L)\ar@{-}[ru]& &}
&\xymatrix{ & & \\
           \phi(L)\ar@{-}[r]&\phi(M)\ar@{-}[d]\ar@{-}[r] &\phi(N) \\
           & &}
&\xymatrix{\phi(L)\ar@{-}[dr] & & \\
           &\phi(M)\ar@{-}[d]\ar@{-}[dr] & \\
           & &\phi(N)}
\end{array}
$$
\caption{The see-saw (or teeter-totter) property.}
\end{figure}
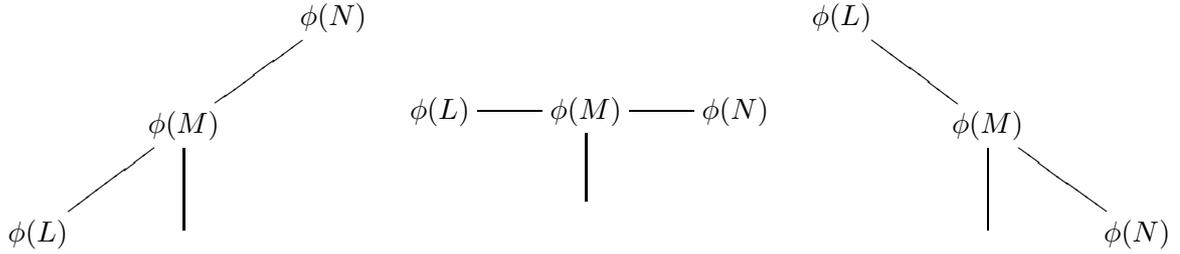

Stability functions in the physics literature are induced by a central charge $Z$. We recall this notion here from \cite{B}:
\begin{ex}
A \textit{linear stability function} on an abelian category $\mathcal{A}$ is given by a central charge, that is, a group homomorphism $Z:K(\mathcal{A})\rightarrow \mathbb{C}$ on the Grothendieck group $K(\mathcal{A})$ such that for all $0\neq M\in \mathcal{A}$ the complex number $Z(M)$ lies in the strict upper half-plane 
$$\mathbb{H}=\{r \cdot\text{exp}(i\pi\phi):r>0\text{ and } 0\leq\phi < 1\}.$$

Given such a central charge $Z:K(\mathcal{A})\rightarrow \mathbb{C}$, the phase of an object $0\neq M\in \mathcal{A}$ is defined to be 
$$\phi(M)=(1/\pi)\text{arg}Z(M)$$
\end{ex}

Clearly the phase function $\phi: Obj(\mathcal{A})\rightarrow [0,1]$ satisfies the see-saw property.

\begin{defi}\cite[Definition 1.5 and 1.6]{Ru} \label{defss}
Let  $\phi:Obj(\mathcal{A})\rightarrow \mathcal{P}$ be a stability function on $\mathcal{A}$. An object $0\neq M$ in $\mathcal{A}$ is said to be \textit{$\phi$-stable} (or \textit{$\phi$-semi-stable}) if every nontrivial subobject $L\subset M$ satisfies $\phi(L)< \phi(M)$ ( or $\phi(L)\leq \phi(M)$, respectively).
\end{defi}

When there is a fixed stability function $\phi: Obj(\A)\ra \P$, 
we  simply refer to $\phi$-stable modules as stable. 
\begin{rmk}
Note that, due to the see-saw property, one can equally define the semi-stable objects as those objects $M$ whose quotient objects $N$ satisfy  $\phi(N)\geq \phi(M)$.
\end{rmk}

The following proposition is an analogue of King's result \cite{Ki} that semi-stable modules of a fixed slope form an abelian category (that is, a wide subcategory of $\mathcal{A}$):
\begin{prop}\label{fix-slope}
Let $p\in \mathcal{P}$ be fixed. Then the full subcategory 
$$\mathcal{A}_{p}=\{0\} \cup \{M\in {\mathcal A}: M\text{ is semi-stable and }\phi(M)=p\}$$
is an abelian category.
\end{prop}

\begin{proof}
It is sufficient to show that $\mathcal{A}_{p}$ is closed under taking kernels and cokernels.
Let $f:M\to N$ be a morphism in $\mathcal{A}_{p}$.  If $f$ is zero or an isomorphism, the result follows at once.  Otherwise, consider the following short exact sequences in $\mathcal A$
$$ 0 \to \ker f \to M \to \im f \to 0$$
$$ 0 \to \im f \to N \to \coker f \to 0$$
where all these objects are nonzero. The semistability of $M$ implies $\phi(\im f) \geq \phi(M)=p$, while the semistability of $N$ implies $\phi(\im f)\leq \phi(N)=p$.  Consequently $\phi(\im f)=p$. The see-saw property applied to the two exact sequences yields $\phi(\ker f)=p$ and $\phi(\coker f)=p$. 

Moreover, every subobject $L$ of $\ker f$ is a subobject of $M$, thus $\phi(L)\le \phi(M)=\phi(\ker f)$. Therefore $\ker f$ is semi-stable and belongs to $\mathcal{A}_p$. Dually we show that $\coker f$ also belongs to $\mathcal{A}_p$, finishing the proof. 
\end{proof}

\begin{rmk}\label{simplestable}
It is easy to see that the stable objects with phase $ p$ are exactly the simple objects of the abelian category $\mathcal{A}_p$.
\end{rmk}

Motivated by Proposition \ref{fix-slope} one might also consider the following subcategories:
\begin{itemize}
\item $\mathcal{A}_{\geq p}=\{0\} \cup \{M\in {\mathcal A}: M\text{ is semi-stable and }\phi(M)\geq p\}$

\item $\mathcal{A}_{\leq p}=\{0\} \cup \{M\in {\mathcal A}: M\text{ is semi-stable and }\phi(M)\leq p\}$

\item $\mathcal{A}_{< p}=\{0\} \cup \{M\in {\mathcal A}: M\text{ is semi-stable and }\phi(M)<p\}$

\item $\mathcal{A}_{> p}=\{0\} \cup \{M\in {\mathcal A}: M\text{ is semi-stable and }\phi(M)> p\}$
\end{itemize}

However it turns out that these subcategories do not have good properties in general. We study in Proposition \ref{eqdeftorsion} below which construction leads to well-behaved subcategories, such as torsion classes.

The following theorem from \cite{Ru} implies that morphisms between semi-stable objects respect the order induced by $\phi$, that is, $\Hom_{\A}(M,N)=0$ whenever  $M,N$ are semi-stable and  $\phi(M)>\phi(N)$.  Observe that parts (b) and (c) are direct consequences of Proposition~\ref{fix-slope}.

\begin{theorem}\cite[Theorem 1]{Ru}\label{homzero}
Let $f: M \to N$ be a nonzero morphism in $\mathcal{A}$ between semi-stable objects $M,N$ such that $\phi(M) \ge \phi(N)$. Then
\begin{itemize}
\item[(a)] $\phi(M) = \phi(N)$;
\item[(b)] If $N$ is stable then $f$ is an epimorphism;
\item[(c)] If $M$ is stable then $f$ is a monomorphism.
\end{itemize}
\end{theorem}

We note the following consequence, which is also clear from the fact that stable objects are simple in the full subcategory $\mathcal{A}_{p}$.

\begin{cor}\label{noniso}
Let $M,N\in\A$ be two nonisomorphic stable objects having the same phase. Then $\emph{Hom}_{\A}(M,N)=0$.
\end{cor}

\begin{rmk}
As observed in \cite{Ru}, Theorem \ref{homzero} implies that stable objects are Schur objects when $\A$ is a Hom-finite $k$-category over an algebraically closed field $k$. Here $M$ is called Schur object when End$(M) \simeq k$. This implies in particular that stable objects are indecomposable. In fact, it is easy to see that stable objects are always indecomposable, for any abelian category $\A$.
\end{rmk}

\subsection{Harder-Narasimhan filtration} From now on, we assume that the abelian category $\mathcal A$ is a \textit{length category}, that is, each object $M$ admits a filtration 
$$0=M_0\subsetneq M_1\subsetneq M_2\subsetneq \dots \subsetneq M_{l-1} \subsetneq M_l=M$$
such that the quotients $M_i/M_{i-1}$ are simple. In particular, $\mathcal{A}$ is both noetherian and artinian. For  a finite dimensional $k$-algebra $A$ over a field  $k$, the category mod\;$A$ of finitely generated $A$-modules is a length category.

We borrow the following terminology from \cite{B}, however the concept was already used in \cite{Ru}.

\begin{defi}\label{defMDQ}
Let $\mathcal{A}$ be an abelian length category and let $M\in \mathcal A$.
\begin{enumerate}[label=(\alph*)]
\item A pair $(N,p)$ consisting of an object $N\in\mathcal{A}$ and an epimorphism $p:M\rightarrow N$ is said to be a \textit{maximally destabilizing quotient} of $M$ if $\phi(M)\geq\phi(N)$ and every other epimorphism $p':M\rightarrow N'$ satisfies $\phi(N')\geq\phi(N)$, and moreover, if  $\phi(N)=\phi(N')$, then the morphism $p'$ factors through $p$. 
\item A pair $(L,i)$ consisting of an object $L\in\mathcal{A}$ and a monomorphism $i:L\rightarrow M$ is a \textit{maximally destabilizing subobject } of $M$ if $\phi(M)\leq\phi(L)$ and every other monomorphism $i':L'\rightarrow M$ satisfies $\phi(L')\leq\phi(L)$, and moreover, if  $\phi(L)=\phi(L')$ then the morphism $i'$ factors through $i$.
\end{enumerate}
\end{defi}

We sometimes omit the epimorphism $p$ when referring to a maximally destabilizing quotient, similarly for  maximally destabilizing subobjects.

\begin{lem}\label{MDQunicity}
Let  $\phi:Obj(\mathcal{A})\rightarrow \mathcal{P}$ be a stability function on $\mathcal{A}$ and let $0\neq M$ be an object in $\mathcal{A}$. Then:
\begin{enumerate}[label=(\alph*)]
\item[(a)] The maximally destabilizing object $(N,p)$ of $M$ is semi-stable and unique up to isomorphism;
\item[(b)] The maximally destabilizing subobject $(L,i)$ of $M$ is semi-stable and unique up to isomorphism.
\end{enumerate}
\end{lem}

\begin{proof}
Consider a maximally destabilizing quotient $(N,p)$ of $M$ and let $\widetilde{N}$  be a quotient of $N$ with quotient map $p': N\ra \widetilde{N}\ra 0$. The composition $p'p$ is an epimorphism from $M$ to $\widetilde{N}$, hence $\phi(\widetilde{N})\geq \phi(N)$ by the definition of maximally destabilizing quotient. Therefore $N$ is semi-stable. 

Suppose that $M$ admits two different maximally destabilizing quotients $(N,p)$ and $(N',p')$. Then $\phi(N)\leq\phi(N')$ and $\phi(N')\leq\phi(N)$, thus they have the same phase. Definition \ref{defMDQ} implies the existence of morphisms $f: N\rightarrow N'$ and $f':N'\rightarrow N$ such that $p=f'p'$ and $p'=fp$, thus  $p=f'p'=f'fp$. Hence the composition $f'f$ is the identity map in $N$ because $p$ is an epimorphism. Likewise, $ff'$ is the identity map in $N'$, and so $N$ and $N'$ are isomorphic, which finishes the proof of statement (a). Statement (b) is shown dually.
\end{proof}

 
The following theorem from \cite{Ru} implies in particular that  every object admits a maximally destabilizing quotient and a maximally destabilizing subobject.

\begin{theorem}\cite[Theorem 2, Proposition 1.13]{Ru}\label{HN}
Let $\mathcal{A}$ be an abelian length category with a stability function $\phi:Obj\mathcal{A}\rightarrow \mathcal{P}$, and let $M$ be a nonzero object in $\mathcal{A}$.  Up to isomorphism, $M$ admits a unique \textit{Harder-Narasimhan filtration}, that is a filtration 
$$0=M_0\subsetneq M_1\subsetneq M_2\subsetneq \dots \subsetneq M_{n-1} \subsetneq M_n=M$$
such that 
\begin{enumerate}[label=(\alph*)]
\item the quotients $F_i=M_i/M_{i-1}$ are semi-stable, 
\item $\phi(F_n)<\phi(F_{n-1})<\dots<\phi(F_2)<\phi(F_1)$.
\end{enumerate}
Moreover, $F_1=M_1$ is the maximally destabilizing subobject of $M$ and $F_n=M_n/M_{n-1}$ is the maximally destabilizing quotient of $M$.
\end{theorem}

For further use, it is also worthwhile to recall the following weaker version of a result from Rudakov.

\begin{theorem}\cite[Theorem 3]{Ru}\label{JH}
Let $\mathcal{A}$ be an abelian length category with a stability function $\phi:Obj\mathcal{A}\rightarrow \mathcal{P}$, and let $M$ be a semi-stable object in $\mathcal{A}$.  There exists a filtration 
$$0=M_0\subsetneq M_1\subsetneq M_2\subsetneq \dots \subsetneq M_{n-1} \subsetneq M_n=M$$
such that 
\begin{enumerate}[label=(\alph*)]
\item the quotients $G_i=M_i/M_{i-1}$ are stable, 
\item $\phi(M)=\phi(G_{n-1})=\dots=\phi(G_2)=\phi(G_1)$.
\end{enumerate}
Moreover, the Jordan-H\"older property holds, in the sense that the set $\{G_i\}$ of factors is uniquely determined up to isomorphism.
\end{theorem}

\subsection{Torsion pairs}\label{Sect:torsion} 

It is well-known that a subcategory $\mathcal{T}$ of $\mathcal{A}$ is the torsion class of a torsion pair if and only if $\mathcal{T}$ is closed under quotients and extensions. Dually, a subcategory $\mathcal{F}$ of $\mathcal{A}$ is the torsion-free class of a torsion pair if and only if $\mathcal{F}$ is closed under subobjects and extensions. 

In this section, we show that a stability function $\phi: Obj \A \to \mathcal P$  induces a torsion pair $(\mathcal{T}_p, \mathcal{F}_p)$ in $\mathcal{A}$ for every $p\in\mathcal{P}$, with
$$\mathcal{T}_{p}=\{M\in Obj(\mathcal A) \ | \ \phi(M')\geq p, \text{ where $M'$ is the maximally destabilizing quotient of }M\}$$
$\mathcal{F}_p=\{M\in Obj(\mathcal{A}) \ | \ \phi(M'')<p\text{ for the maximally destabilizing subobject $M''$  of $M$}\}$
\bigskip

The following proposition not only shows that $\T_p$ is a torsion class, but also gives a series of equivalent characterizations. 

\begin{prop}\label{eqdeftorsion}
Let $\phi: \A\ra \P$ be a stability function and consider the full subcategory $\T_p$ of $\A$ defined above. Then:
\begin{enumerate}[label=(\alph*)]
\item $\T_p$ is a torsion class;
\item $\mathcal{T}_{p}=\text{\emph{Filt}}(\mathcal{A}_{\geq p})$;
\item $\T_p=\text{\emph{Filt}}(\text{\emph{Fac}}\mathcal{A}_{\geq p})$;
\item $\T_p=\{M\in\A \ : \phi(N)\geq p \text{ for every quotient $N$ of $M$}\}$.
\end{enumerate}
\end{prop}

\begin{proof}
\textit{1.} We need to show that $\mathcal{T}_{p}$ is closed under extensions and quotients. 

To show that $\mathcal{T}_{p}$ is closed under extensions, suppose that  
$$0\rightarrow L\overset{f}\rightarrow M\overset{g}\rightarrow N\rightarrow 0$$ is a short exact sequence in $\mathcal{A}$ with $L, N\in\mathcal{T}_{p}$.  Let $(M',p_M)$ be the maximally destabilizing quotient of $M$. Then we can construct the following commutative diagram.
$$\xymatrix{
 0\ar[r]& L\ar[r]^f\ar[d] &M\ar[r]^g\ar[d]^{p_M} &N\ar[r]\ar[d] &0 \\
 0\ar[r]& \im (p_Mf)\ar[r]^{f'}\ar[d] & M'\ar[r]^{g'}\ar[d] &\coker f'\ar[r]\ar[d] &0 \\
  & 0 & 0 & 0 & }$$
Let $(L',p_L)$ and $(N',p_N)$ be the maximally destabilizing quotients of $L$ and $N$ respectively.

If $\im (p_Mf)=0$, then there exists an epimorphism $h:N\rightarrow M'$, and it follows from the definition of $N'$ that $\phi(M')\geq\phi(N')\geq p$.
Else, it follows from the semistability of $M'$ that $\phi(\im (p_Mf))\leq \phi(M')$. Moreover, $\phi(\im(p_Mf))\geq\phi(L')\geq p$ since $L'$ is a maximally destabilizing quotient. Consequently $\phi(M')\geq p$ and $\mathcal{T}_P$ is closed under extensions.

To show that $\mathcal{T}_p$ is closed under quotients, suppose that $f:M\rightarrow N$ is an epimorphism with $M\in\mathcal{T}_p$. Let $(M',p_M)$ and $(N',p_N)$ be the maximally destabilizing quotients of $M$ and $N$ respectively. Then $p_Nf:M\rightarrow N'$ is an epimorphism and it follows from the definition of $M'$ that $\phi(N')\geq\phi(M')\geq p$. Hence $N \in \T_p$. 
This proves that $\T_p$ is a torsion class.

\textit{2. and 3.} Clearly, $\text{Filt}(\mathcal{A}_{\geq p})\subseteq \text{Filt}(\Fac\mathcal{A}_{\geq p})$. On the other hand, it follows from \cite[Proposition 3.3]{DIJ} that $\text{Filt}(\Fac\mathcal{A}_{\geq p})$ is the smallest torsion class containing $\mathcal{A}_{\geq p}$. As $\mathcal{A}_{\geq p}\subseteq\mathcal{T}_p$, we get $\text{Filt}(\Fac\mathcal{A}_{\geq p})\subseteq\mathcal{T}_p$. 

It thus remains to show that $\mathcal{T}_p\subseteq \text{Filt}(\mathcal{A}_{\geq p})$. Let $M\in\mathcal{T}_{p}$, and let $M'$ be a maximally destabilizing quotient of $M$.  By definition of $\mathcal{T}_p$, we have that $\phi(M')\geq p$. Therefore we can consider the Harder-Narasimhan filtration of $M$ and Theorem \ref{HN} implies that $M\in\text{Filt}(\mathcal{A}_{\geq p})$. Hence $\T_p\subseteq \text{Filt}(\mathcal{A}_{\geq p})\subseteq\text{Filt}(\Fac(\A_{\geq p}))\subseteq \T_p$.

\textit{4.} Let $M\in\mathcal{T}_p$, and suppose that $M'$ is its maximally destabilizing quotient.  By definition of the maximally destabilizing quotient, every quotient $N$ of $M$ is such that $\phi(N)\geq\phi(M')\geq p$.  Thus $\mathcal{T}_p\subseteq \{M\in\A \ : \ \phi(N)\geq p \text{ for every quotient $N$ of $M$}\}$.  The reverse inclusion is immediate.
\end{proof}

The following result is the dual statement for the torsion-free class $\mathcal{F}_p$.

\begin{prop}\label{eqdeftorsionfree}
Let $\phi: \A\ra \P$ be a stability function and consider the full subcategory $\F_p$ of $\A$ defined as $$\mathcal{F}_{p}=\{M\in Obj(\mathcal A) \ | \ \phi(M'')< p, \text{ where $M''$ is the maximally destabilizing subobject of }M\}.$$ Then:
\begin{enumerate}[label=(\alph*)]
\item $\F_p$ is a torsion free class;
\item $\mathcal{F}_{p}=\text{\emph{Filt}}(\mathcal{A}_{< p})$; 
\item $\F_p=\text{\emph{Filt}}(\text{\emph{Fac}}\mathcal{A}_{< p})$.
\item $\F_p=\{M\in\A \ : \ \phi(N)< p \text{ for every subobject $N$ of $M$}\}$.
\end{enumerate}
\end{prop}

Now were are able to  prove the main result of this section.

\begin{prop}\label{torsionpair}
Let $p\in\mathcal{P}$.  Then $(\mathcal{T}_p, \mathcal{F}_p)$ is a torsion pair in $\mathcal{A}$.
\end{prop}
\begin{proof}
We first show that $\Hom_{\mathcal{A}}(\mathcal{T}_p, \mathcal{F}_p)=0$.  Suppose that $f\in\Hom_{\A}(M,N)$, where $M\in\T_p$ and $N\in\F_p$.  Let $M'$ be the maximally destabilizing quotient of $M$ and $N'$ be the maximally destabilizing subobject of $N$.  Then $\im f$ is a quotient of $M$ and a subobject of $N$. So, if $f\neq 0$, it follows from the definitions of $M'$ and $N'$ that $\phi(\im f)\geq \phi(M')\geq p$ and $\phi(\im f)\leq\phi(N')<p$, a contradiction.  Thus $f=0$ and $\Hom_{\mathcal{A}}(\mathcal{T}_p, \mathcal{F}_p)=0$.

For the maximality, suppose for instance that $\Hom_\A(\T_p, N)=0$. If $N'$ is the maximally destabilizing subobject of $N$, it follows that $\Hom_\A(\T_p, N')=0$, and thus $\phi(N')<p$ by definition of $\T_p$. Consequently, $N\in\F_p$.  We show in the same way that $\Hom_\A(M, \F_p)=0$ implies $M\in\T_p$, showing the maximality.
\end{proof}

\section{Maximal green sequences in abelian length categories}

In the previous section we discussed how a stability function $\phi: \A \to \P$ induces a torsion pair $(\T_p,\F_p)$ in $\A$ for each phase $p\in\P$. Moreover, it is easy to see that if $p\leq q$ in $\mathcal{P}$, then $\T_p\supseteq\T_q$ and $\F_p\subseteq\F_q$.  Since $\P$ is totally ordered, every stability function  $\phi$ yields thus a (possibly infinite) chain of torsion classes in $\A$.
In this section we are mainly interested in the different torsion classes induced by $\phi$. We therefore define, for a fixed stability function $\phi: \A \to \P$, an equivalence relation on $\mathcal P$ by $p \sim q$ when $\mathcal{T}_p =\mathcal{T}_q$ and consider the equivalence classes $\P/\sim$.  

Of particular importance is the case where the chain of equivalence classes $\P/\sim$ is finite, not refinable, and represented by elements $p_0 > \ldots > p_m \in \P$  such that $\T_{p_0} = \{0\}$ and $\T_{p_m} = \A$:

\begin{defi}
A \textit{maximal green sequence} in $\A$ is a finite sequence of torsion classes $0=\X_0\subsetneq \X_1 \subsetneq \dots \subsetneq \X_{n-1}\subsetneq \X_n=\A$ such that for all $i\in\{1, 2, \dots, n\}$, the existence of a torsion class $\X$ satisfying $\mathcal{X}_i\subseteq\mathcal{X}\subseteq\mathcal{X}_{i+1}$ implies $\mathcal{X}=\mathcal{X}_i$ or $\mathcal{X}=\mathcal{X}_{i+1}$.
\end{defi}

Our aim is to establish conditions when the chain of torsion classes induced by a stability function is a maximal green sequence.

Observe first that if $\phi: \A \to \P$ is a stability function and the totally ordered set $\P$ has a maximal element $\overline{p}$, then $\T_{\overline{p}}$ is the minimal element in the chain of torsion classes induced by $\phi$.  The following lemma determines when $\T_{\overline{p}}=\{0\}$.

\begin{lem}\label{supremum}
Let $\phi: \A \to \P$ be a stability function.  
\begin{enumerate}[label=(\alph*)]
\item If $\P$ has a maximal element $\overline{p}$, then $\T_{\overline{p}}\neq \{0\}$ if and only if $\overline{p}\in\phi(\A)$.
\item If the set of equivalence classes $\P/\sim$ is finite and the maximal object of $\P$ does not belong to the image of $\phi$, then there exists some $p\in \P$ such that $\T_p=\{0\}$.
\end{enumerate} 
\end{lem}
\begin{proof}
(a) Suppose that $\overline{p}$ is a maximal element in $\P$. 
If $\T_{\overline{p}}\neq \{0\}$, then there exists a nonzero object $M$ in $\T_{\overline{p}}$. If $M'$ is the maximally destabilizing quotient of $M$, we know that $\phi(M')\geq \overline{p}$. Since $\overline{p}$ is the maximal element of $\P$, we have   $\phi(M')=\overline{p}$ and thus $\overline{p}\in \phi(\A)$.

Conversely, if $\phi(M)=\overline{p}$, then it follows from the maximality of $\overline{p}$ that $\phi(L)\leq\phi(M)=\overline{p}$ for every nontrivial subobject $L$ of $M$. Thus $M$ is a semi-stable object, whence $M\in\A_{\overline{p}}\subset\T_{\overline{p}}$.

(b) By assumption, the chain of torsion classes induced by $\phi$ is finite, say $$
\T_{p_0}\subsetneq \T_{p_1} \subsetneq \cdots \subsetneq \T_{p_n}.
$$

If $\T_{p_0}\neq\{0\}$, choose a nonzero object $M$ in $\T_{p_0}$.  Let $M'$ be the maximally destabilizing quotient of $M$, thus $M'\in\T_{p_0}$ and $\phi(M')\geq p_0$. Since the maximal object of $\P$ does not belong to the image of $\phi$, there exists a $p\in \P$ with $p>\phi(M')$.  It follows that $M'\notin \T_{p}$, while $\T_p\subseteq \T_{p_0}$, contradicting the minimality of $\T_{p_0}$.  Thus $\T_{p_0}=\{0\}$.  
\end{proof}

Following Engenhorst \cite{E14}, we call a stability function $\phi : \A \rightarrow \mathcal{P}$ \textit{discrete at $p$} if two stable objects $M_1, M_2$ satisfy $\phi(M_1)=\phi(M_2)=p$ precisely when $M_1$ and $M_2$ are isomorphic in $\A$. Moreover, we say that $\phi$ is \textit{discrete} if $\phi$ is discrete at $p$ for every $p\in\P$. 

\begin{prop}\label{mutation}
Let $\phi: \A \to \P$ be a stability function, and let $p,q\in\P$ such that $\T_p\subsetneq\T_q$.  Then the  following statements are equivalent:
\begin{enumerate}[label=(\alph*)]
\item There is no $r\in\P$ such that $\T_p\subsetneq\T_r\subsetneq\T_q$, and $\phi$ is discrete at every $q'$ with $q'\sim q$.
\item There is no torsion class $\T$ such that $\T_p\subsetneq\T\subsetneq\T_q$. 
\end{enumerate}
\end{prop}
\begin{proof}
(a) implies (b): Suppose that $\T$ is a torsion class such that $\T_{p}\subsetneq\T\subseteq \T_{q}$. Then there exists an object $M\in\T\setminus\T_{p}$. Let $M'$ be the maximally destabilizing quotient of $M$. 
Then $\phi(M')\geq q$ because $M\in \mathcal{T}\subsetneq\mathcal{T}_{q}$.  Consequently, $\T_{\phi(M')}\subseteq\T_q$.  On the other hand, $\phi(M')<p$ because $M\not\in \mathcal{T}_{p}$. Moreover, since $M'$ is semi-stable by Theorem~\ref{HN}, $M'\in\T_ {\phi(M')}\setminus\T_p$.  Consequently, $\T_p\subsetneq \T_{\phi(M')}\subseteq \T_q$.  It thus follows from our assumption that $\T_{\phi(M')}=\T_q$. 

Now Theorem~\ref{JH} implies the existence of a stable object $M''$ such that $\phi(M'')=\phi(M')$, which is unique since $\phi$ is discrete.  Using Theorem~\ref{JH} again, $M'$ can be filtered by $M''$. In particular $M''$ is a quotient of $M$, and thus $M''\in\T$. 

Consider a stable object $X$ in $\A_{\geq \phi(M'')}$. In particular $X\in\T_{\phi(M'')}=\T_q$.  If $\phi(X)=\phi(M'')$, then $X$ is isomorphic to $M''$ by the discreteness, and $X\in \T$.  Else $\phi(X)>\phi(M'')$, and $M''\in\T_{\phi(M'')}\setminus\T_{\phi(X)}$. Therefore, $\T_{\phi(X)}\subsetneq \T_{\phi(M'')}=\T_q$, which implies by assumption, that $\T_{\phi(X)}\subseteq \T_p\subseteq \T$. In particular, $X\in\T$.  Since $\T$ is a torsion class, this implies that $\A_{\geq \phi(M'')}\subseteq \T$, and furthermore $$\T_q=\T_{\phi(M'')}=\text{Filt}(\A_{\geq\phi(M'')})\subseteq \T.$$  This shows $\T_q=\T$.

(b) implies (a): The fact that there is no $r\in\P$ such that $\T_p\subsetneq\T_r\subsetneq\T_q$ is immediate.  To show that $\phi$ is discrete, assume that there exist two nonisomorphic stable objects $M$ and $N$ such that $\phi(M)=\phi(N)=q'$, with $q'\sim q$. Consider the set $\T=\text{Filt}(\A_{\geq p}\cup \{N\})$. We will show that $\T$ is a torsion class such that $\T_p\subsetneq \T\subsetneq \T_q$, a contradiction to our hypothesis.

First, because $\T_p\subsetneq\T_q=\T_{q'}$, we have $q<p$. Since $\T_p=\text{Filt}(\A_{\geq p})$, we have $N\notin \T_p$, so $\T_p\subsetneq \T$.  Furthermore, it follows from Theorem\ref{homzero} and Corollary~\ref{noniso} that $\Hom_{\A}(\T,M)=0$, implying $M\notin \text{Filt}(\A_{\geq p}\cup \{N\})$.  Since $M\in\T_q$, this shows $\T\subsetneq\T_q$.  Thus $\T_p\subsetneq\T\subsetneq\T_q$.    

We now show that $\T=\text{Filt}(\A_{\geq p}\cup \{N\})$ is a torsion class, that is, $\T$ is closed under extensions and quotients. By definition, $\T$ is closed under extensions.  To show that $\T$ is closed under quotients, suppose that $$T\rightarrow T'\rightarrow 0$$ is an exact sequence in $\A$ and $T\in\T$.  

If $T\in\T_p$, then $T'\in\T_p$ since $\T_p$ is a torsion class and therefore $T'\in\T$.

Else, $T\in\T\setminus\T_p$.  Let $Q$ be the maximally destabilizing quotient of $T$. Since $\T\notin\T_p$, we have $\phi(Q)<p$.  Moreover, $\phi(Q)\geq q$ since $T\in\T\subsetneq\T_q$.  Consequently, $q\leq\phi(Q)<p$, and it follows from our hypothesis that $\phi(Q)=q$ (otherwise $\T_p\subsetneq\T_{\phi(Q)}\subsetneq \T_q$).  So $Q\in\T_q=\T_{q'}$. This shows in particular that $q=q'$.  Indeed, if $q<q'$, then the fact that $Q$ is semi-stable leads to $Q\in\T_q\notin\T_{q'}$, a contradiction.  Similarly, if $q'<q$, then $N\in\T_{q'}\notin\T_q$, again a contradiction.  So $q=q'$, and consequently $Q,N\in\A_q$.

Now, suppose that
$$
0=T_0\subsetneq T_1\subsetneq T_2\subsetneq \cdots \subsetneq T_{n-1}\subsetneq T_n=T
$$
is a Harder-Narasimhan filtration of $T$, as in Theorem~\ref{HN}.  In particular, $Q\cong T/T_{n-1}$ and $$q=\phi(Q)<\phi(T_{n-1}/T_{n-2}) <\cdots< \phi(T_{2}/T_{1})<\phi(T_1/T_0).$$  Consequently, $T_i/T_{i-1}\in\A_p$ for all $i\leq n-1$, while $\phi(Q)=q$.  It thus follows from the fact that $T\in\text{Filt}(\A_p\cup\{N\})\setminus\T_p$ that $Q\in\text{Filt}(\{N\})$. In particular, $Q\in\T$.

Now, let $Q'$ be the maximally destabilizing quotient of $T'$.  Since $Q$ is the maximally destabilizing quotient of $T$, we have $\phi(Q')\geq \phi(Q)$.  If $\phi(Q')>\phi(Q)$, then $\phi(Q')\geq p$, and $T'\in\T_p\subsetneq \T$.  Else, $\phi(Q')=\phi(Q)$, and it follows from the fact that $Q$ is the maximally destabilizing quotient of $T$ that the epimorphism from $T$ to $Q'$ factors through $Q$, and thus there exists an epimorphism $f:Q\rightarrow Q'$ in $\A$, and thus in $\A_q$.

  Recall from  Proposition~\ref{fix-slope} that $\A_q$ is an abelian category whose stable objects coincide with the simple objects by Remark~\ref{simplestable}.  Consequently, it follows from the existence of the epimorphism $f:Q\rightarrow Q'$ and the fact that $Q$ is filtered by the stable object $N$ that $Q'\in\text{Filt}(\{N\})$.  

Let
$$
0=T'_0\subsetneq T'_1\subsetneq T'_2\subsetneq \cdots \subsetneq T'_{m-1}\subsetneq T'_m=T'
$$
be the Harder-Narasimhan filtration of $T'$. Then $Q'\cong T'/T'_{n-1}$ and $$q=\phi(Q')<\phi(T'_{n-1}/T'_{n-2}) <\cdots< \phi(T'_{2}/T'_{1})<\phi(T'_1/T'_0).$$
 Consequently, $T'_i/T'_{i-1}\in\A_p$. Since $Q'$ is filtered by $N$, this implies that $T'\in\text{Filt}(\A_p\cup\{N\})=\T$. This finishes the proof.
\end{proof}

We are now able to characterize the stability functions inducing maximal green sequences in $\A$.

\begin{theorem}\label{maximalgreensequences}
Let $\phi:\A\ra \P$ be a stability function. Suppose that $\P$ has no maximal element, or that the maximal element of $\P$ is not in $\phi(\A)$. Then $\phi$ induces a maximal green sequence if and only if $\phi$ is a discrete stability function inducing finitely many equivalent classes on $\P/\sim$.
\end{theorem}
\begin{proof}
Suppose that $\phi$ induces a maximal green sequence, say 
$$
\{0\}=\T_{p_0}\subsetneq \T_{p_1} \subsetneq \cdots \subsetneq \T_{p_n}=\A.
$$
In particular, there are only finitely many equivalence classes in $\P/\sim$.
Moreover, it follows from Proposition~\ref{mutation} that $\phi$ is discrete.

  Conversely, suppose that $\phi$ is a discrete stability function inducing finitely many equivalent classes on $\P/\sim$.  So we get a (complete) chain of torsion classes
$$
\T_{p_0}\subsetneq \T_{p_1} \subsetneq \cdots \subsetneq \T_{p_n}.
$$
induced by $\phi$.
 The discreteness of $\phi$ implies by Proposition~\ref{mutation} that this chain of torsion classes is maximal. Moreover, it follows from Lemma~\ref{supremum} that $\T_{p_0}=\{0\}$. It remains to show that $\T_{p_n}=\mod A$. If $M\in\mod A$ but $M\notin\T_{p_n}$, then the maximally destabilizing quotient $M'$ of $M$ satisfies $\phi(M')<p_n$. Since $M'\in \T_{\phi(M')}$, it follows that $\T_{p_n}\subsetneq \T_{\phi(M')}$, a contradiction to the maximality of $\T_{p_n}$.  So $\T_{p_n}=\mod A$. 
\end{proof}

\section{Hyperplane arrangements and cone complex}\label{sect:conecomplex}

From now on, we will consider the case when $\A$ is the category of finitely generated modules $\mod A$ over a finite dimensional algebra $A$ over an algebraically closed field $k$. The Grothendieck group of $A$ is denoted by $K_0(A)$, and the class in the Grothendieck group of an $A$-module $M$ is denoted by $[M]$, which is identified with the dimension vector of $M$. Moreover, we assume that $rk(K_0(A))=n$. The Auslander-Reiten translation in $\mod A$ is denoted by $\tau$. 
Unless otherwise specified, we consider only nonzero modules $M$.

As we said before, Rudakov's notion of stability generalizes the stability conditions introduced by King in \cite{Ki}, which is the following.

\begin{defi}\cite[Definition 1.1]{Ki}
Let $\theta:K_0(A)\otimes\mathbb{R} \ra \mathbb{R}$ be an element of $(\mathbb{R}^n)^*$, the dual vector space of $\mathbb{R}^n$. An $A$-module $M\in\mod A$ is called \textit{$\theta$-stable} (or \textit{$\theta$-semi-stable}) if $\theta(M)=0$ and $\theta(L)<0$ ($\theta(L)\leq 0$, respectively) for every proper submodule $L$ of $M$. 
\end{defi}

The stability condition of King depends on two variables: the module and the linear functional. Fixing a module leads to the following definition.

\begin{defi}
The \textit{stability space of an $A-$module $M$} is $$\D(M)=\{\theta\in(\mathbb{R}^n)^* : M \text{ is $\theta$-semi-stable}\}.$$
\end{defi}

Note that $M$ induces an element $\varphi_M$ in the double dual $(\mathbb{R}^n)^{**}$ of $\mathbb{R}^n$ defined by $\varphi_M(\theta)=\theta(M)$ for all $\theta\in(\mathbb{R}^n)^{*}$. Therefore $\D(M)$ is included in $\ker \varphi_M$ which is an hyperplane of $(\mathbb{R}^n)^*$. Hence $\D(M)$ has codimension at least 1. 

\begin{defi}
Let $\A$ be an abelian category having exactly $n$ non isomorphic simple objects and $M$ be an object of $\A$. Then the stability space $\D(M)$ of $M$ is said to be a \textit{wall} when $\D(M)$ has codimension one. 
We refer in this case to $\D(M)$ as the wall defined by $M$.
\end{defi}

Note that not every $\theta$ belongs to the stability space $\D(M)$ for some nonzero module $M$. For instance, is easy to see that 
$$\theta(x_1,x_2, \dots, x_n)=\sum_{i=1}^n x_i$$
is an example of such a functional for every algebra $A$. This leads to the following definition.

\begin{defi}
Let $A$ be an algebra such that $rk(K_0(A))=n$ and
$$\mathfrak{R}=(\mathbb{R}^n)^*\setminus\overline{\bigcup\limits_{\substack{ M\in \mod A}}\mathcal{D}(M)}$$
be the maximal open set of all $\theta$ having no $\theta$-semi-stable modules other that the zero object. Then a dimension $n$ connected component $\mathfrak{C}$ of $\mathfrak{R}$ is called a \textit{chamber}. 
\end{defi}

We will consider the following example throughout the paper.

\begin{ex}\label{runningexample}
Consider the path algebra $\mathbb{A}_2=kQ$ of the quiver $Q=\xymatrix{1\ar[r]& 2}$, having the following Auslander-Reiten quiver. 

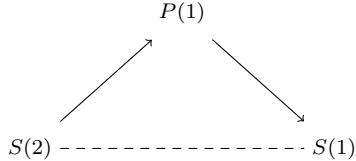
\begin{center}
			\begin{tikzpicture}[line cap=round,line join=round ,x=2.0cm,y=1.8cm]
				\clip(-1.2,-0.1) rectangle (1.2,1.1);
					\draw [->] (-0.8,0.2) -- (-0.2,0.8);
					\draw [dashed] (-0.8,0.0) -- (0.8,0.0);
					\draw [->] (0.2,0.8) -- (0.8,0.2);
				
				\begin{scriptsize}
					\draw[color=black] (-1,0) node {$S(2)$};
					\draw[color=black] (0,1) node {$P(1)$};
					\draw[color=black] (1,0) node {$S(1)$};
				\end{scriptsize}
			\end{tikzpicture}
		\end{center}

We can see the wall and chamber structure of $\mathbb{A}_2$ in the Figure \ref{w&cA2}.

\begin{figure}
			\begin{center}
				\begin{tikzpicture}[line cap=round,line join=round,>=triangle 45,x=3.0cm,y=3.0cm]
					\clip(-1.5,-1.5) rectangle (1.5,1.5);
						\draw (0.,0.) -- (0.,1.5);
						\draw (0.,0.) -- (0.,-1.5);
						\draw [domain=-1.5:0.0] plot(\x,{(-0.-0.*\x)/-1.});
						\draw [domain=0.0:1.5] plot(\x,{(-0.-1.*\x)/1.});
						\draw [domain=0.0:1.5] plot(\x,{(-0.-0.*\x)/1.});
					\begin{scriptsize}
						\draw[color=black] (0,1.25) node[anchor= south west] {$\mathfrak{D}(S(1))$};
						\draw[color=black] (-0.2,-1.4) node {$\mathfrak{D}(S(1))$};
						\draw[color=black] (-1.25,-0.1) node {$\mathfrak{D}(S(2))$};
						\draw[color=black] (1.15,-1.4) node {$\mathfrak{D}(P(1))$};
						\draw[color=black] (1.25,0.12) node {$\mathfrak{D}(S(2))$};
					\end{scriptsize}
				\end{tikzpicture}
			\end{center}
			\caption{Wall and chamber structure for $\mathbb{A}_2$}
            \label{w&cA2}
		\end{figure}
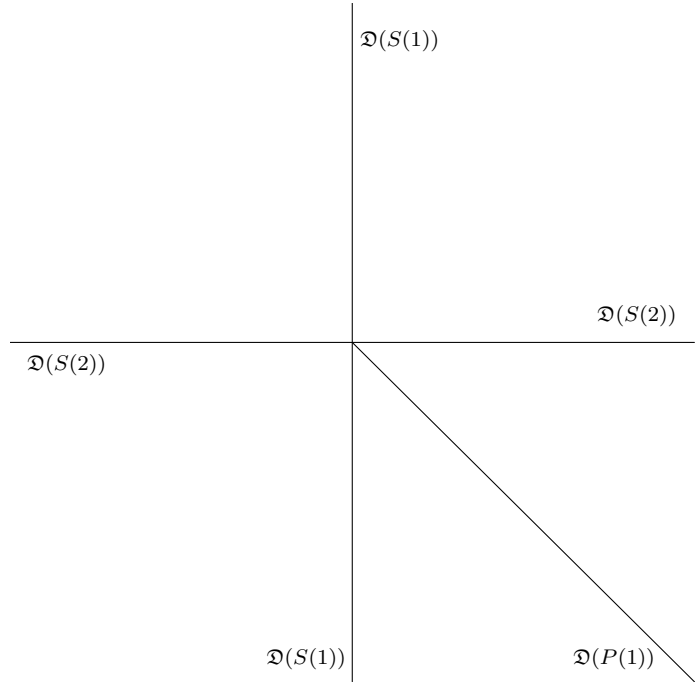
\end{ex}

In this section we study how the simplicial complex introduced by Demonet, Iyama and Jasso in \cite{DIJ} can be embedded into the wall and chamber structure of a given algebra.

In the first subsection we study the $\theta$-stable modules, for some linear functional $\theta$ associated to the positive cone of the $g$-vectors of some $\tau$-rigid pair $(M',P')$. In the second subsection we fix a $\tau$-tilting pair $(M,P)$ and we show that the positive cone of its $g$-vectors induces a chamber $\Ch_{(M,P)}$ and we give a complete description of the walls surrounding $\Ch_{(M,P)}$. Finally, in the third section, we study the relation between chambers and torsion classes.

\subsection{On $\theta_{\alpha(M,P)}$-semi-stable modules}

One aim of the paper is to give a geometrical characterization of the maximal green sequences in the module category $\mod A$ for an algebra $A$. We achieve this using the $\tau$-tilting theory introduced by Adachi, Iyama and Reiten in \cite{AIR}. In particular we consider the $\tau$-rigid and $\tau$-tilting pairs, see the following definition.

\begin{defi}\cite[Definition 0.1 and 0.3]{AIR}\label{list}
		Let $A$ an algebra, $M$ an $A$-module and $P$ an $A$-projective module. The pair $(M,P)$ is said \textit{$\tau$-rigid} if:

		\begin{itemize}
			\item $\Hom_{A}(M,\tau M)=0$;
			\item $\Hom_{A}(P,M)=0$.
		\end{itemize}
		
		Moreover, we say that $(M,P)$ is \textit{$\tau$-tilting} (or \textit{almost $\tau$-tilting}) if $|M|+|P|=n$ (or $|M|+|P|=n-1$, respectively).
	\end{defi}

The $\tau$-rigid pairs play a key role in the connection between the maximal green sequences and the stability condition introduced by King. On one hand, as we will see in Corollary \ref{SVMmod}, every maximal green sequence is determined by a suitable set of $\tau$-tilting pairs. Conversely, the $g$-vectors of every $\tau$-rigid pair $(M,P)$ induce some linear functionals $\theta_{\alpha(M,P)}$ for which we give a complete description of the category of $\theta_{\alpha(M,P)}$-semi-stable modules. The definition of the $g$-vectors is the following. 

\begin{defi}
Let $M$ be an $A$-module. If $$P_1\longrightarrow P_0\longrightarrow M\longrightarrow 0$$ is a minimal projective presentation of $M$, where $P_0=\bigoplus\limits_{i=1}^n P(i)^{c_i}$ and $P_1=\bigoplus\limits_{i=1}^n P(i)^{c'_i}$, then we define the $g$-vector of $M$ to be 
$$g^M=(c_1-c'_1, c_2-c'_2,\dots, c_n-c'_n).$$
\end{defi}

\begin{rmk}
Decomposing $A=P(1)\oplus \dots \oplus P(n)$  as a sum of indecomposable $A$-modules,  the set $$\{g^{P(1)}, \dots, g^{P(n)}\}$$ yields the canonical basis of $\mathbb{Z}^n$. 
\end{rmk}

The following two results give us important properties of the $g$-vectors of $\tau$-rigid and $\tau$-tilting pairs. The first of them  extends the previous remark to all $\tau$-tilting pairs.

\begin{theorem}\cite[Theorem 5.1]{AIR}\label{LI}
Let $(M,P)$ be a $\tau$-tilting pair such that $M=\bigoplus_{i=1}^kM_i$ and $P=\bigoplus_{j=k+1}^n P_j$. Then the set
$$\{g^{M_1},\dots, g^{M_k}, -g^{P_{k+1}}, \dots, -g^{P_n}\}$$
forms a basis for $\mathbb{Z}^n$.
\end{theorem}

\begin{theorem}\cite[Theorem 5.5]{AIR}\label{inj-g-vecteurs}
Let $M$ and $N$ be two $\tau$-rigid $A$-modules. Then $g^M=g^N$ if and only if $M$ is isomorphic to $N$.
\end{theorem}

In this paper, when we write $\langle -,-\rangle$ we are referring to the canonical inner product in $\mathbb{R}^n$ which is defined by
$$\langle v,w\rangle=\sum_{i=1}^n v_iw_i$$
for every $v$ and $w$ in $\mathbb{R}^n$. It is well-known that $\Phi:\mathbb{R}^n\ra (\mathbb{R}^n)^*$ defined by $\Phi(v)(-)=\theta_{v}(-)=\langle v,-\rangle$ is an anti-isomorphism of vector spaces. From now on, in order to illustrate our results with figures, we will identify in the examples $\theta\in(\mathbb{R}^n)^*$ with the vector $v=\Phi^{-1}(\theta)$. 

Given $M\in\mod A$, we denote by $\theta_M$ the functional induced by its $g$-vector, that is $\theta_M(-)=\Phi(g^M)=\langle g^M,-\rangle$. The next result of Auslander and Reiten  \cite{AR1} shows that if we apply $\theta_M$ to the dimension vector $[N]$ of a given module $N$, then we can get explicit homological information.

\begin{theorem}[\cite{AR1}, Theorem 1.4.(a)]\label{formula}
Let $M$ and $N$ be arbitrary modules over an algebra $A$. Then we have
$$\theta_M ([N])=\langle g^{M},[N]\rangle=\dim_k(\emph{Hom}_A(M,N))-\dim_k(\emph{Hom}_A(N,\tau_A M))$$
\end{theorem}

For the sake of simplicity, instead of writing $\dim_k(\Hom_A(M,N))$ we abbreviate it by $\emph{hom}(M,N)$. 
Also, we denote $M^{\perp}=\{X\in\mod A: \Hom_A(M,X)=0 \}$ and, dually, $^{\perp}M=\{Y\in\mod A: \Hom_A(Y,M)=0\}$. 

The following formula is a particular case of \cite[Theorem 1.4]{AR1}, but plays a central role in the remaining part of the paper.

\begin{cor}\label{tauformula}
Let $A$ be an algebra, $(M,P)$ be a $\tau$-rigid pair and $N$ be an arbitrary $A$-module. Then
$$\theta_{(M,P)}([N])=\langle g^{M}-g^{P},[N]\rangle=hom_A(M,N))-hom_A(N,\tau_A M))-hom_A(P,N)).$$
\end{cor}

\begin{proof}
It suffices to apply Theorem \ref{formula} twice and notice that $\tau P=0$.
\end{proof}

\begin{rmk}
Note that for a $\tau$-tilting module $M$ of $\text{pd} M\leq 1$ we have that $\langle g^{M},[N]\rangle$ coincides with the classical Euler-Ringel form. 
\end{rmk}

Recall that if $(\T,\F)$ is a torsion pair in mod$A$, then for every $A$-module $M$ exists a so-called canonical short exact sequence
$$0\ra tM\ra M\ra M/tM\ra 0$$
where $tM\in\T$ and $M/tM\in\F$  are unique up to isomorphism. 

Given a set of linearly independent vectors $L=\{v_1, \dots, v_t\}$, we define the \textit{positive cone $\mathcal{C}_L$} of $L$ as $$\mathcal{C}_L=\left\{\sum\limits_{i=1}\limits^{t} \alpha_iv_i : \alpha_i>0 \text{ for every $1\leq i\leq t$}\right\}.$$ 

Let $(M,P)$ be a $\tau$-rigid pair where $M=\bigoplus\limits_{i=1}^kM_i$ and $P=\bigoplus\limits_{j=k+1}\limits^tP_j$ are the decomposition of $M$ and $P$ as sums of indecomposable modules, respectively. Then the positive cone $\Co_{(M,P)}$ of the $\tau$-rigid pair $(M,P)$ is the positive cone generated by the $g$-vectors of the indecomposable direct factors of $M$ and the opposite of the $g$-vectors of $P$, that is $$\mathcal{C}_{(M,P)}=\left\{\sum\limits_{i=1}\limits^{k} \alpha_ig^{M_i}-\sum\limits_{j=k+1}\limits^{t} \alpha_j g^{P_j} : \alpha_i>0 \text{ for every $1\leq i\leq t$}\right\}.$$ Moreover the generic vector $\sum\limits_{i=1}\limits^{k} \alpha_ig^{M_i}-\sum\limits_{j=k+1}\limits^{t} \alpha_j g^{P_j}\in\Co_{(M,P)}$ is denoted by $\alpha(M,P)$.

Now that the notation is fixed, we start stating and proving the results of this section.

\begin{lem}\label{lemtrace}
Let $(M,P)$ be a $\tau$-tilting pair and let $\alpha(M,P)\in\Co_{(M,P)}$. Then $\theta_{\alpha(M,P)}([tN])\geq 0$ for every $N\in \emph{mod} A$.

\end{lem}

\begin{proof}
First, we have that Fac$M$ is a torsion class by \cite[Theorem 5.10]{ASfuncfiniteFac}. Therefore $(\Fac M, M^{\perp})$ is a torsion pair. 

Then, for every $N\in\mod A$, there exist a canonical short exact sequence 
$$0\rightarrow tN\rightarrow N\rightarrow N/tN\rightarrow 0$$
where $tN\in \Fac M$ and $N/tN\in M^{\perp}$.

This implies $hom_A(M,tN)>0$ and $\emph{hom}_A(tN,\tau M)=0$. Moreover, since $tN\in \Fac M$, there exists an epimorphism $p: M^k\ra tN$ for some $k\in\mathbb{N}$. Hence the projectivity of $P$ implies that every morphism $f:P\ra tN$ factors through $M^k$. Therefore $f=0$ because $\Hom_A(P,M)=0$ and thus $\emph{hom}_A(P,tN)= 0$. Therefore we have that 
$$\theta_{\alpha(M,P)}([tN])=\sum^{k}_{i=1}\alpha_i\emph{hom}(M_i,tN)-\sum_{i=1}^k\alpha_i\emph{hom}(tN,\tau M_i)-\sum_{j=k+1}^t\alpha_j\emph{hom}(P_j,tN)$$
$$=\sum^{k}_{i=1}\alpha_i\emph{hom}(M_i,tN)>0$$
finishing the proof. 
\end{proof}

\begin{lem}\label{catsemi-stables}
Let $(M,P)$ be a $\tau$-rigid pair and let $\alpha(M,P)\in\Co_{(M,P)}$. Then $N$ is a $\theta_{\alpha(M,P)}$-semi-stable module if and only if $N\in{^{\perp}\tau M}\cap M^{\perp}\cap P^{\perp}$.
\end{lem}

\begin{proof}
Suppose that $N\in{ ^{\perp}\tau M}\cap M^{\perp}\cap P^{\perp}$. Then $\emph{hom}(M,N)=\emph{hom}(N,\tau_A M)=\emph{hom}(P,N)=0$. Therefore $\theta_{\alpha(M,P)}([M])=0$ by Corollary \ref{tauformula}. Moreover, for every submodule $L$ of $N$ we have that $\emph{hom}(M,L)=0$ because $\emph{hom}(M,N)=0$. Likewise $\emph{hom}(P,L)=0$. Then $\theta_{\alpha(M,P)}([L])\leq -\sum_{i=1}^k\alpha_i\emph{hom}(N,\tau_A M_i)\leq 0$. Hence $N$ is a $\theta_{\alpha(M,P)}$-semi-stable module. 

Now, suppose that $N$ is a $\theta_{\alpha(M,P)}$-semi-stable module. Then  $\theta_{\alpha(M,P)}([L])\leq 0$ for every submodule $L$ of $N$. In particular $\theta_{\alpha(M,P)}([tN])\leq 0$. Then  $\theta_{\alpha(M,P)}([tN])= 0$ by Lemma \ref{lemtrace}. So
$$\sum_{i=1}^k\alpha_i\emph{hom}(M_i,tN)-\sum_{i=1}^k\alpha_i\emph{hom}(tN,\tau_A M_i)-\sum_{j=k+1}^t\alpha_j\emph{hom}(P_j,tN)=0.$$
Since $tN\in \Fac M$ and $M$ is $\tau$-rigid, we get $\sum_{i=1}^k\alpha_i\emph{hom}(tN,\tau_A M_i)=0$.  Similarly, since $hom(P,M)=0$ and $P$ is projective, we get $\sum_{j=k+1}^t\alpha_j\emph{hom}(P_j,tN)=0$. Consequently, $\sum_{i=1}^k\alpha_i\emph{hom}(M_i,tN)=0$, and $tN=0$.

Therefore $\emph{hom}(M,N)=0$. Then $\theta_{\alpha(M,P)}([N])=0$ implies that $\emph{hom}(N,\tau M)=\emph{hom}(P,N)=0$. Hence $N\in {^{\perp}\tau M}\cap M^{\perp}\cap P^{\perp}$, finishing the proof. 
\end{proof}

With this result in hand we can prove one of our main results. 

\begin{theorem}\label{stablemodcat}
Let $(M,P)$ be a $\tau$-rigid pair and let $\alpha(M,P)\in\Co_{(M,P)}$. Then the category of $\theta_{\alpha(M,P)}$-semi-stable modules is equivalent to the module category of an algebra $C$. Moreover there are exactly $rk(K_0(A))-|M|-|P|$ nonisomorphic $\theta_{\alpha(M,P)}$-stable modules.

\end{theorem}

\begin{proof}
The first part of the statement follows directly from Lemma \ref{catsemi-stables} and \cite[Theorem 3.8]{J}. For the moreover part, it suffices to note that $rk(K_0(C))=rk(K_0(A))-|M|-|P|$ by construction of $C$ (see \cite[Setting 3.1]{J}). Also the functor $F$ of \cite[Theorem 3.8]{J} induces a bijection between the isomorphism classes of $\theta_{\alpha(M,P)}$-stable modules and the isomorphism classes simple $C$-modules. Therefore the number of nonisomorphic $\theta_{\alpha(M,P)}$-stable modules coincide with the rank of $K_0(C)$ the Grothendieck group of mod$C$, finishing the proof.
\end{proof}

\subsection{Geometrical consequences of Theorem \ref{stablemodcat}}

In this subsection we use Theorem \ref{stablemodcat} to show in Proposition \ref{chambers} that every $\tau$-tilting pair induces a chamber and give a complete description of them in Corollary \ref{walls&chambers}.

\begin{prop}\label{chambers}
Let $A$ be an algebra. Then the positive cone $\Co_{(M,P)}$ of every $\tau$-tilting pair $(M,P)$ defines a chamber $\mathfrak{C}_{(M,P)}$ in the wall and chamber structure of $A$.
\end{prop}

\begin{proof}
Let $(M,P)$ be a $\tau$-tilting pair and $\alpha(M,P)\in\Co_{(M,P)}$. Then Theorem \ref{stablemodcat} implies that the category of $\theta_{\alpha(M,P)}$-semi-stable modules consist only on the zero object. Therefore $\theta_{(M,P)}$ belongs to a chamber $\mathfrak{C}$. Moreover every vector in $\Co_{(M,P)}$ belongs to the same chamber $\mathfrak{C}$ because $\Co_{(M,P)}$ is connected. Then the positive cone of $(M,P)$ is a subset of a some chamber $\mathfrak{C}$. 

Moreover, for every vector $\beta(M,P)$ in the boundary $\partial\Co_{(M,P)}$ of $\Co_{(M,P)}$ there exists an $1\leq i\leq n$ such that $\beta_i=0$. Hence, Theorem \ref{stablemodcat} yields the existence of a $\theta_{\beta(M-P)}$-stable module implying that the positive cone coincides with $\mathfrak{C}$. 
\end{proof} 

In view of the previous result, the positive cone $\Co_{(M,P)}$ of a $\tau$-tilting pair $(M,P)$ will be referred as the chamber induced by $(M,P)$ and denoted by $\Ch_{(M,P)}$.

Another immediate consequence of Theorem \ref{stablemodcat} is the following.

\begin{cor}
Let $(M,P)$ be an almost $\tau$-tilting pair. Then $\Co_{(M,P)}$ has codimension 1 and is included in a wall of the wall and chamber structure of \emph{mod}$A$.
\qed
\end{cor}

The previous corollary shows that the positive cones of almost $\tau$-tilting pairs are included in a wall. However it gives us no information about the module defining that wall. To construct it explicitly, we recall (see \cite[Theorem 2.18]{AIR}) that for every almost $\tau$-tilting pair $(M,P)$ there exist exactly two $\tau$-tilting pairs $(M_1,P_1)$ and $(M_2,P_2)$ such that $M$ is a direct factor of each $M_i$ and $P$ is a direct factor of each $P_i$. Moreover one of them, that we can suppose without loss of generality to be $M_1$, is such that $\text{Fac}M=\text{Fac}M_1$. On the other hand, $\text{Fac} M\subsetneq\text{Fac}M_2$. Therefore there exists an indecomposable $\tau$-rigid module $M'$ such that $M_2=M'\oplus M$. 

\begin{prop}\label{walls}
Let $(M,P)$ be an almost $\tau$-tilting pair let $M'$ as above. Then $\Co_{(M,P)}$ is included in the wall $\D(N)$, where $N$ is the cokernel of the right \emph{add}$M$-approximation of $M'$. 
\end{prop}

\begin{proof}
Let $M'$ be as in the hypothesis and consider the short exact sequence induced by the right add$M$-approximation as follows.
$$0\rightarrow tM'\overset{i}\rightarrow M'\overset{p}\rightarrow N\rightarrow 0$$
Note that, by construction, Hom$_A(M,N)=0$. 

On the other hand, $\Hom_A(P,M')=\Hom_A(M',\tau M)=0$ because $(M\oplus M',P)$ is a $\tau$-tilting pair. The projectivity of $P$ implies that $\Hom_A(P,N)=0$. Moreover, any morphism $f:N\ra \tau M$ induces a morphism $fp:M'\ra \tau M$, hence $\Hom_A(N,\tau M)=0$.

Since $N$ belongs to $ {^{\perp}\tau M}\cap P^{\perp}\cap M^{\perp}$, we know that $N$ is a $\theta_{\alpha(M,P)}$-stable module. Hence the positive cone induced by an almost $\tau$-tilting pair is included in the stability space $\D(N)$. Moreover, the positive cone has codimension 1, implying that $\D(N)$ is actually a wall.
\end{proof}

As a corollary, we get the following result, which completely describes  the chamber induced by a $\tau$-tilting pair $(M,P)$.

\begin{cor}\label{walls&chambers}
Let $(M,P)$ be a $\tau$-tilting pair. Then $(M,P)$ induces a chamber $\Ch_{(M,P)}$ having exactly $n$ walls $\{\D(N_1), \dots, \D(N_n)\}$. Moreover, the stable modules $\{N_1, \dots, N_n\}$ can be calculated explicitly from $(M,P)$.
\end{cor}

\begin{proof}
Suppose that $M=\bigoplus_{i=1}^k M_i$ and $P=\bigoplus_{j=k+1}^tP_j$. The fact that $(M,P)$ induces a chamber follows from Proposition~\ref{chambers}. Moreover, we know from Proposition \ref{walls} that each of the $n$ almost $\tau$-tilting pairs that can be derived from $(M,P)$ are included in a wall. It was also shown in Proposition \ref{walls} an explicit way to calculate the stable modules defining those walls. Finally, the fact that every $\D(N_i)$ does not coincide with $\D(N_j)$ if $i\neq j$ follows from the fact that the set of $g$-vectors $\{g^{M_1}, \dots, g^{M_k}, -g^{P_{k+1}}, \dots, -g^{P_{t}}\}$ forms a basis, as shown in \cite[Theorem 5.1]{AIR}.
\end{proof}

\begin{rmk}
Not every wall is generated by the positive cone of some almost $\tau$-tilting pair. For instance, take the hereditary algebra of the Kronecker quiver $\xymatrix{1\ar@<0.5 ex>[r]\ar@<-0.5 ex>[r]&2}$ and consider the wall $\mathfrak{D}=\{\lambda(1,-1):\lambda>0\}$. In this case there is no $\tau$-rigid module $M$ such that $g^M=(1,-1)$.
Note that the wall $\mathfrak{D}$ has no adjacent chamber, but rather is a limit of walls given by preprojective (or preinjective) modules.

Also the number of positive cones defined by almost $\tau$-tilting pairs included in a given wall is not constant. For instance $\D(\rep{1\\22})=\{\lambda(g^{\rep{11\\222}}): \lambda\in\mathbb{R}_{\geq 0}\}$ while $\D(\rep{2})=\{\lambda(g^{\rep{1\\22}}): \lambda\in\mathbb{R}_{\geq 0}\}\cup\{\lambda(-g^{\rep{1\\22}}): \lambda\in\mathbb{R}_{\geq 0}\}$
\end{rmk}

The next example is intended to be an illustration of our previous propositions.

\begin{ex}
Consider the wall and chamber structure of the algebra $\mathbb{A}_2$ as we did in Example \ref{runningexample}. In Figure \ref{tauw&cA_2} we can see how positive cones of $\tau$-tilting pairs coincides with chambers and how positive cones of almost $\tau$-tilting pairs are included in walls of the wall and chamber structure of $\mathbb{A}_2$.

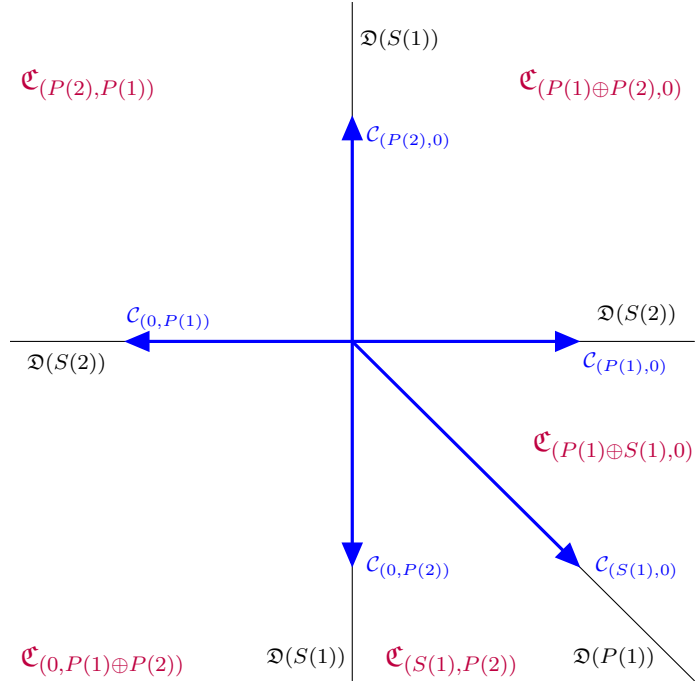
\begin{figure}
			\begin{center}
				\begin{tikzpicture}[line cap=round,line join=round,>=triangle 45,x=3.0cm,y=3.0cm]
					\clip(-1.5,-1.5) rectangle (1.5,1.5);
						\draw (0.,0.) -- (0.,1.5);
						\draw (0.,0.) -- (0.,-1.5);
						\draw [domain=-1.5:0.0] plot(\x,{(-0.-0.*\x)/-1.});
						\draw [domain=0.0:1.5] plot(\x,{(-0.-1.*\x)/1.});
						\draw [domain=0.0:1.5] plot(\x,{(-0.-0.*\x)/1.});
						\draw [->,line width=1.2pt,color=blue] (0.,0.) -- (1.,0.);
						\draw [->,line width=1.2pt,color=blue] (0.,0.) -- (0.,1.);
						\draw [->,line width=1.2pt,color=blue] (0.,0.) -- (-1.,0.);
						\draw [->,line width=1.2pt,color=blue] (0.,0.) -- (0.,-1.);
						\draw [->,line width=1.2pt,color=blue] (0.,0.) -- (1.,-1.);
						\draw[color=purple] (1.5,1.25) node[anchor=north east] {$\Ch_{(P(1)\oplus P(2),0)}$};
						\draw[color=purple] (-1.5,1.25) node[anchor=north west] {$\Ch_{(P(2),P(1))}$};
						\draw[color=purple] (1.55,-0.35) node[anchor=north east] {$\Ch_{(P(1)\oplus S(1), 0)}$};
						\draw[color=purple] (0.1,-1.3) node[anchor=north west] {$\Ch_{(S(1),P(2))}$};
						\draw[color=purple] (-1.5,-1.3) node[anchor=north west] {$\Ch_{(0,P(1)\oplus P(2))}$};
					\begin{scriptsize}
						\draw[color=black] (0,1.25) node[anchor= south west] {$\mathfrak{D}(S(1))$};
						\draw[color=black] (-0.2,-1.4) node {$\mathfrak{D}(S(1))$};
						\draw[color=black] (-1.25,-0.1) node {$\mathfrak{D}(S(2))$};
						\draw[color=black] (1.15,-1.4) node {$\mathfrak{D}(P(1))$};
						\draw[color=black] (1.25,0.12) node {$\mathfrak{D}(S(2))$};
						\draw[color=blue] (1.2,-0.1) node {$\Co_{(P(1),0)}$};
						\draw[color=blue] (0.25,0.9) node {$\Co_{(P(2),0)}$};
						\draw[color=blue] (-0.8,0.1) node {$\Co_{(0,P(1))}$};
						\draw[color=blue] (0.25,-1.0) node {$\Co_{(0,P(2))}$};
						\draw[color=blue] (1.25,-1.0) node {$\Co_{(S(1),0)}$};
					\end{scriptsize}
				\end{tikzpicture}
			\end{center}
			\caption{Wall and chamber structure for $\mathbb{A}_2$}
            \label{tauw&cA_2}
		\end{figure}
\end{ex}

\subsection{Torsion classes associated to chambers}

As we proved in the last subsection, every $\tau$-tilting pair $(M,P)$ defines a chamber $\Ch_{(M,P)}$ in the wall and chamber structure. In this subsection we associate to a given chamber a torsion class $\T_{\Ch}$ and we show that $\Fac M=\T_{\Ch_{(M,P)}}$ if the chamber $\Ch_{(M,P)}$ is defined by a $\tau$-tilting pair $(M,P)$.

Recall that Bridgeland associated in \cite[Lemma 6.6]{B16} a torsion class $\T_{\theta}$ to every functional $\theta$ as follows.
$$\T_{\theta}=\{M\in \mod A:\theta(N)>0 \text{ for every nonzero quotient $N$ of $M$}\}$$

The next proposition give us a natural torsion class associated to a given chamber.

\begin{prop}
Let $\Ch$ be a chamber and consider $\T_{\Ch}=\bigcap\limits_{\theta\in\Ch}\T_{\theta}$. Then $\T_{\Ch}$ is a torsion class.
\end{prop}

\begin{proof}
This follows directly from \cite[Lemma 6.6]{B16} and the fact that an arbitrary intersection of torsion classes is again a torsion class.
\end{proof}

\begin{prop}\label{eqtorsionch}
Let $(M,P)$ be a $\tau$-tilting pair, $\Ch_{(M,P)}$ its induced chamber and $\D(N_1)$, $\D(N_2)$, $\dots, \D(N_n)$ be the walls surrounding $\Ch_{(M,P)}$. Then $\T_{\Ch_{(M,P)}}=\emph{Fac}M$. Moreover, if we define $\mathcal{N}$ as
$$\mathcal{N}=\{N_i:\theta(N_i)>0 \text{ for all $\theta\in\Ch_{(M,P)}$}\}$$
then $\Fac M=T(\mathcal{N})$, where $T(\mathcal{N})$ is the minimal torsion class containing $\mathcal{N}$.
\end{prop}

\begin{proof}
First we prove that $\Fac M=\T_{\Ch_{(M,P)}}$. Indeed, consider $(\Fac M, M^\perp)$ be the torsion pair defined by $(M,P)$ and let $N\in\mod A$. Take the canonical short exact sequence 
$$0\ra tN \ra N \ra N/tN \ra 0$$
associated to $N$. Then $N\in\T_{\Ch_{(M,P)}}$ if and only if $N/tN=0$ by Lemma \ref{lemtrace}. Therefore $N\in\T_{\Ch_{(M,P)}}$ if and only if $N\in\Fac M$. 

Now we prove that $\Fac M=T(\mathcal{N})$. Suppose that $M$ and $P$ can be written as $M=M_1\oplus\dots\oplus M_k$ and $P=P_{k+1}\oplus\dots \oplus P_n$, where all $M_i$ and $P_j$ are indecomposable.

Note that the sector of each wall $\D(N_i)$ that lies in the boundary of $\Ch_{(M,P)}$ corresponds to $\Co_{(M/M_i,P)}$ if $1\leq i\leq k$ or $\Co_{(M,P/P_i)}$ if $k+1\leq i\leq t$.

Consider the exact sequences $M'_i\overset{f_i}\rightarrow M_i\rightarrow C_i\rightarrow 0$, where $f_i$ is the right $(\add(M/M_i))$-approximation of $M_i$. Is easy to see that $C_i\in \Fac M$ for every $i$. We have that $N_i=C_i$ if and only if $\Fac (\mu_{M_i}M)\subsetneq \Fac M$ by Proposition \ref{walls}. Moreover Lemma \ref{lemtrace} implies that $\theta_{\alpha(M,P)}(N_i)>0$ if and only if $N_i=C_i$. Then \cite[Lemma 3.7]{DIJ} implies that $\Fac M=\T_{\Ch_{(M,P)}}$.
\end{proof}

\begin{rmk}\label{sametorsion}
Note that, if a chamber $\Ch$ is induced by a $\tau$-tilting pair $(M,P)$, then $\T_{\theta}=\Fac M$ for every $\theta\in\Ch$ by Theorem \ref{stablemodcat}.
\end{rmk}

Following the terminology introduced in \cite{DIJ} we say that an algebra $A$ is \textit{$\tau$-tilting finite} if there exists a finite number of indecomposable $\tau$-rigid modules. 

\begin{cor}
Let $A$ be an algebra. Then the function $\Ch$ mapping a $\tau$-tilting pair $(M,P)$ to its corresponding chamber $\Ch_{(M,P)}$ is injective. Moreover, if $A$ is $\tau$-tilting finite then $\Ch$ is also surjective.
\end{cor}
\begin{proof}
This follows directly from the fact that $\Ch_{(M,P)}=\text{Fac}M$ and \cite[Theorem 2.7]{AIR}. The moreover part follows from \cite[Theorem 1.5]{DIJ}
\end{proof}

We saw already the existence of walls in the wall and chamber structure of an algebra which are not induced by any almost $\tau$-tilting pairs. However, in the examples that we computed, these walls do not bound a chamber. Hence  the following conjecture.

\begin{conj}
Let $A$ be an algebra. Then the every chamber in the wall and chamber structure of $A$ is induced by a $\tau$-tilting pair. 
\end{conj}

\section{Maximal green sequences induced by paths}

In this section, we study maximal green sequences in module categories using the techniques developed in the previous section. In the first subsection we introduce the notion of \textit{green paths} on the wall and chamber structure (see Definition \ref{greenpaths} below) of an algebra $A$. We show that every green path induces a stability function which is compatible with the wall and chamber structure. Later, we use green paths to prove that if the $\tau$-tilting pair $(M',P')$ is obtained from $(M,P)$ by mutation  and $\Fac M\subset \Fac M'$, then there is no torsion class strictly contained between $\Fac M$ and $\Fac M'$. We further give a characterization of the green paths that induce a maximal green sequence and we show that every maximal green sequence in $\mod A$ arises that way. Finally, we end the paper giving a new proof of the well known fact that the Markov algebra does not admit any maximal green sequence. 

\subsection{Green paths}

Let $A$ be an algebra and $\gamma:[0,1]\rightarrow \mathbb{R}^n$ be a continuous function such that $\gamma(0)=(1,\dots, 1)$ and $\gamma(1)=(-1,\dots,-1)$. If we fix an $A$-module $M$, $\gamma(t)$ induces a continuous function $\rho_M:[0,1]\rightarrow \mathbb{R}$ defined as $\rho_M(t)=\theta_{\gamma(t)}([M])$. Note that $\rho_M(0)>0$ and $\rho_M(1)<0$. Therefore, for every module there is at least one $t\in(0,1)$ such that $\theta_{\gamma(t)}([M])=0$. This leads us to the following definition of green paths:

\begin{defi}\label{greenpaths}
A \textit{green path in $\mathbb{R}^n$} is a continuous function $\gamma: [0,1]\ra\mathbb{R}^n$ such that $\gamma(0)=(1,\dots, 1)$, $\gamma(1)=(-1,\dots,-1)$, and such that for every $M\in\mod A$ there is exactly one $t_M\in[0,1]$ such that $\theta_{\gamma(t_M)}([M])=0$.
\end{defi}

\begin{rmk}
Note that, by definition, green paths can pass through the intersection of walls, which is not allowed in the definition of Bridgeland's \textit{$\D$-generic paths} (see \cite[Definition 2.7]{B16}) nor Engenhorst's  \textit{discrete paths} (see \cite{E14}).
\end{rmk}

Another key difference between the green paths and the other paths cited above is that green paths take account of crossing of all hyperplanes, not only  the walls. In the next proposition we show that we can recover the information of  crossings from the stability structure induced by the path. 

\begin{lem}\label{darkside}
Let $\gamma$ be a green path. Then $\theta_{\gamma(t)}([M])<0$ if and only if $t>t_M$.
\end{lem}

\begin{proof}
This is a direct consequence of the definition of green path and the fact that the function $\rho_M$ induced by $\gamma$ and $M$ is continuous. 
\end{proof}

\begin{prop}\label{paths}
Let $A$ be an algebra such that $rk(K_0(A))=n$. Then every green path $\gamma:[0,1]\to \mathbb{R}^n$ induces a stability structure $\phi_{\gamma}:\emph{mod} A\to [0,1]$ defined by $\phi_{\gamma}(M)=t_M$, where $t_M$ is the unique element in $[0,1]$ such that $\theta_{\gamma(t_M)}(M)=0$. Moreover $M$ is $\phi_{\gamma}$-semi-stable if and only if $M$ is $\theta_{\gamma(t_M)}$-semi-stable. 
\end{prop}
 
\begin{proof}
Let $\gamma$ be a green path in $\mathbb{R}^n$. First, note that $\phi_{\gamma}$ is a well defined function by the definition of green path. We want to show that $\phi_\gamma$ induces a stability structure in $\mod A$. Since $\gamma$ is continuous, one has $\theta_{\gamma(t)}([M])<0$ if and only if $t>t_M$, and $\theta_{\gamma(t)}([M])>0$ if and only if $t<t_M$. Consider a short exact sequence $$0\ra L\ra M\ra N\ra 0$$ and suppose that $\phi_\gamma([L])<\phi_\gamma([M])$. Then $\theta_{\gamma(t_M)}([M])=0$ and $\theta_{\gamma(t_M)}([L])<0$. Therefore $\theta_{\gamma(t_M)}([N])=\theta_{\gamma(t_M)}([M]-[L])=\theta_{\gamma(t_M)}([M])-\theta_{\gamma(t_M)}([L])>0$. Hence $\phi_{\gamma}([L])<\phi_\gamma([M])<\phi_\gamma([N])$. The other two conditions of the see-saw property are proved in a similar way, which shows that $\phi_\gamma$ is a stability function by Definition \ref{seesaw}.

If $M$ is an $A$-module, then $\theta_{\gamma(t_M)}([M])=0$. Now suppose that $M$ is $\phi_{\gamma}$-semi-stable and $L$ is a proper submodule of $M$. Then $\phi_{\gamma}(L)\leq\phi_{\gamma}(M)$ (i.e., $t_L\leq t_M$). Therefore $\theta_{\gamma(t_M)}([L])\leq 0$ and thus $M$ is $\theta_{\gamma(t_M)}$-semi-stable.

On the other hand, suppose that $M$ is $\theta_{\gamma(t_M)}$-semi-stable and $L$ is a proper submodule of $M$. Then $\theta_{\gamma(t_M)}([L])\leq 0$, hence $t_L\leq t_M$. Therefore $M$ is $\phi_{\gamma}$-semi-stable. 
\end{proof}

As a consequence of the previous proposition we get the following result, in which we use the notations of Subsection~\ref{Sect:torsion} with $\P=[0,1]$.

\begin{prop}\label{green-to-red}
Let $\gamma$ be a green path and let $\phi_{\gamma}$ be the stability structure induced by $\gamma$. Then $\T_0=\emph{mod} A$ and $\T_1=\{0\}$.
\end{prop}

\begin{proof}
Let $\gamma$ be a green path and let $M$ an arbitrary $A$-module. Then $\theta_{\gamma(1)}([M])=\langle (-1, -1,\dots, -1), [M]\rangle<0$. Hence $1\not\in \phi_{\gamma}(\mod A)$. Therefore Lemma \ref{supremum} implies that $\T_1=\{0\}$. Using dual arguments one can prove that $\T_0=\mod A$.
\end{proof}

\begin{cor}\label{maximal-paths}
Let $\gamma$ be a green path and let $\phi_{\gamma}$ be the stability function induced by $\gamma$. Then $\gamma$ induces a maximal green sequence if and only if there is a finite set $\{M_1, \dots, M_n: M_i \text{ is a $\phi_{\gamma}$-stable module}\}$ such that $t_{M_i}\neq t_{M_j}$ when $i\neq j$. 
\end{cor}

\begin{proof}
Without loss of generality we can suppose that $t_{M_i}\leq t_{M_j}$ if $i<j$. It is easy to see that the chain of torsion classes induced by $\gamma$ is finite given our hypothesis. Moreover Proposition \ref{green-to-red} implies that $\T_0=\mod A$ and $\T_1=\{0\}$. Finally, we have that $t_{M_i}<t_{M_j}$ whenever $i<j$, then we have that $\phi_{\gamma}$ is a discrete. Therefore Theorem \ref{maximalgreensequences} implies that $\gamma$ induces a maximal green sequence. 
\end{proof}

\subsection{Green paths and mutations}

Let $(M,P)$ be a $\tau$-tilting pair and let $(M',P')$ be a mutation of $(M,P)$. Then we can suppose without loss of generality that $\Fac M'\subsetneq \Fac M$. It was shown in \cite[Theorem 3.1]{DIJ} that $\Fac M' \subsetneq \T \subsetneq \Fac M$ implies that either $\Fac M=\T$ or $\Fac M'=T$. In this subsection we give an alternative proof of this fact using the the notion of green path. 

\begin{lem}\label{torsionpath}
Let $\gamma$ a green path and $(M,P)$ a $\tau$-tilting pair. Suppose that $\gamma(t_0)\in\Ch_{(M,P)}$. Then $\T_{t_0}=\emph{Fac} M$.
\end{lem}

\begin{proof}
Let $\gamma$ a green path and $\phi_\gamma$ the stability function induced by $\gamma$. Consider $t_0$ as in the statement and let $\T_{t_0}$ be the torsion class associated to it. Then, by definition, we have 
$$\T_{t_0}=\text{Filt}(\{N: N \text{ is a $\phi_\gamma$-stable module such that $t_N\geq t_0$}\})$$
by the Proposition \ref{eqdeftorsion}. Now, let $N$ be a $\phi_\gamma$-stable module such that $N\in \T_{\gamma(t_0)}$ and $N'$ a nontrivial quotient of $N$. Then, we have that $t_{N'}>t_N$ because $N$ is $\phi_\gamma$-stable. Hence $\theta_{\gamma(t_0)}([N'])>0$ by Lemma \ref{darkside}. Therefore $N'\in\T_{\theta_{\gamma(t_0)}}$ by \cite[Lemma 6.6]{B16}. Then $\T_{t_0}\subset \T_{\theta_{\gamma(t_0)}} =\Fac M$.

In the other direction, let $M'\in\Fac M$ and $N'$ be its maximally destabilizing quotient with respect to the stability structure induced by $\gamma$. Then Proposition \ref{eqtorsionch} implies that $\theta_{\gamma(t_0)}([N'])>0$. Therefore $t_{N'}>t_M$ by Lemma \ref{darkside}. Hence $M'\in\T_{t_0}$, which implies that $\Fac M\subset \T_{t_0}$. This finishes the proof. 
\end{proof}

\begin{lem}\label{decrisingfunction}
Let $(M,P)$ and $(M',P')$ be two $\tau$-tilting pairs related by  a mutation such that $\Fac M\subsetneq \Fac M'$ and fix $\alpha=(1,1,\dots, 1)$. Then 
$$\theta_{\alpha(M,P)}([N])\leq\theta_{\alpha(M',P')}([N])$$
for every $N\in \mod A$.
\end{lem}

\begin{proof}
Let $N$ be an $A$-module. Then Corollary \ref{tauformula} implies that 
$$\theta_{\alpha(M,P)}([N])=hom(M,N)-hom(N,\tau M)-hom(P,N)$$
and
$$\theta_{\alpha(M',P')}([N])=hom(M',N)-hom(N,\tau M')-hom(P',N).$$
Then is enough to prove that $hom(M,N)\leq hom(M',N)$, $hom(N,\tau M)\geq hom(N,\tau M')$ and $hom(P,N)\geq hom(P',N)$.

First, let $f\in\Hom_A(M,N)$. Then we can extend $f$ to a morphism from $M'$ to $N$ because $M\in\Fac M'$. This implies $hom(M,N)\leq hom(M',N)$.

Dually we prove that every morphism $hom(N,\tau M)\geq hom(N,\tau M')$ because $\tau M'\in\Sub (\tau M)$.

Finally, is enough to note that $P'$ is a direct summand of $P$ to conclude that $hom(P,N)\geq hom(P',N)$. This proves our claim.
\end{proof}

Now we are able to show the main result of this subsection.

\begin{theorem}\label{norefinements}
Let $(M,P)$ be a $\tau$-tilting pair and let $(M',P')$ be a mutation of $(M,P)$ such that $\Fac M\subsetneq \Fac M'$. If $\T$ is a torsion class $\T$ such that $\Fac M' \subset \T \subset \Fac M$, then either $\T=\Fac M'$ or $\T=\Fac M$.
\end{theorem}

\begin{proof}

Fix $\alpha=(1,1,\dots, 1)$ as we did in the previous lemma.

Note that $\Ch_{(M,P)}$ is an open subset in $\mathbb{R}^n$ with the topology induced by the euclidean metric. Therefore there exist an $\varepsilon>0$ such that $(\varepsilon(-1,-1,\dots, -1)+\alpha(M,P))\in\Ch_{(M,P)}$. Then we define the function $\gamma$ as follows.
$$ \gamma(t)= \begin{cases} 
      (3t+1)(1,1,\dots, 1)+(3t)(\alpha(M',P')) & t\in\left[0,\frac{1}{3}\right]\\
      (3t)(\alpha(M',P'))+(3t-1)(\alpha(M,P)+\varepsilon(-1,-1,\dots, -1)) & t\in\left[\frac{1}{3},\frac{2}{3}\right] \\
      (3t-1)(\alpha(M,P)+\varepsilon(-1,-1,\dots, -1))+(3t-2)(-1,-1,\dots, -1) & t\in\left[\frac{2}{3},1\right]
   \end{cases}$$
We claim that $\gamma$ is a green path. 

First, is obvious that $\gamma(0)=(1,1,\dots, 1)$ and $\gamma(1)=(-1,-1,\dots, -1)$. 

Let $N$ be an $A$-module and consider $\rho_N(t)$ defined by $\rho_{N}(t)=\theta_{\gamma(t)}([N])$. Then we have that $\rho_N(0)>0$, $\rho_N(1)<0$ and Lemma \ref{decrisingfunction} implies that 
$$\rho_N\left(\frac{2}{3}\right)\leq\rho_N\left(\frac{1}{3}\right)+\varepsilon\rho_N(1)<\rho_N\left(\frac{1}{3}\right).$$
There is a unique $t_N\in[0,1]$ such that $\rho_N(t_N)=0$ because $\gamma$ is a piecewise linear function. Hence $\gamma$ is a green path. 

We know  from Proposition \ref{paths} that $\gamma$ induces a stability function $\phi_\gamma$ . Moreover, Lemma \ref{torsionpath} implies that $\T_{\frac{1}{3}}=\Fac M'$ and $\T_{\frac{2}{3}}=\Fac M$. Also, given that $(M,P)$ is a mutation of $(M',P')$, Proposition \ref{chambers} and Corollary \ref{walls} imply that there is a unique $\phi_\gamma$-stable module $N$ such that $\frac{1}{3}\leq \phi_\gamma(N)\leq \frac{2}{3}$. Therefore, if $\T$ is a torsion class such that $\Fac M\subset\T\subset\Fac M'$ then $\Fac M=\T$ or $\Fac M'=\T$ by Proposition \ref{mutation}. This finishes the proof.
\end{proof}

As a first consequence of the previous theorem we give a characterization of maximal green sequences in module categories.

\begin{cor}\label{SVMmod}
Let $A$ be an algebra and $\{(M_i,P_i)\}_{i=0}^n$ a set of $\tau$-tilting pairs such that $(M_0,P_0)=(0,A)$, $(M_n,P_n)=(A,0)$, $(M_i,P_i)=\mu(M_{i-1},P_{i-1})$ and $\Fac M_{i-1}\subsetneq \Fac M_i$ for every $i$. Then $$\{0\}=\Fac M_0\subsetneq \Fac M_1\subsetneq \dots\subsetneq \Fac M_n=\mod A$$ is a maximal green sequence. Moreover every maximal green sequence is of this form. 
\end{cor}

\begin{proof}
The first part of the statement follows directly from Theorem \ref{norefinements}. 

For the second, consider a maximal green sequence $$\{0\}=\T_{0}\subsetneq \T_{1} \subsetneq \cdots \subsetneq \T_{n}=\mod A$$ in $\mod A$. Of course, $\{0\}$ is functorially finite. Thus \cite[Theorem 2.7]{AIR} implies the existence of a $\tau$-tilting pair $(M_0,P_0)$ such that $\Fac M_0=\{0\}$, thus it must be the pair $(0,A)$. We have that there is no torsion class strictly between $\T_0=\Fac M_0$ and $T_1$. Therefore \cite[Theorem 3.1]{DIJ} implies the existence of a $\tau$-tilting pair $(M_1,P_1)$ which is a mutation of $(M_0, P_0)$ such that $\T_1=\Fac M_1$. Inductively, suppose that $\T_{i-1}=\Fac M_{i-1}$. Then there is no torsion class strictly  between $\T_{i-1}=\Fac M_{i-1}$ and $T_i$. Therefore \cite[Theorem 3.1]{DIJ} implies the existence of a $\tau$-tilting pair $(M_i,P_i)$ which is a mutation of $(M_{i-1}, P_{i-1})$ such that $\T_i=\Fac M_i$. Finally this process will eventually stop given that maximal green sequences consist only of finitely many torsion classes. Moreover $\mod A=\T_n=\Fac A$, so clearly $(M_n,P_n)=(A,0)$. This finishes the proof. 
\end{proof}

Now we construct a natural graph associated to the wall and chamber structure of an algebra.

\begin{defi}

Let $A$ be an algebra. We define the quiver $\mathfrak{G}_A$ as follows. 

	\begin{itemize}
		\item The vertices of $\mathfrak{G}_A$ correspond to the chambers in the wall and chamber structure of $A$.
		\item There is an arrow from the vertex associated to $\Ch_1$ to the vertex associated to $\Ch_2$ if $\T_{\Ch_2}\subsetneq\T_{\Ch_1}$ and there is no torsion class $\T$ such that $\T_{\Ch_2}\subsetneq\T\subsetneq\T_{\Ch_1}$.
	\end{itemize}
\end{defi}

With this definition in hand, the following result is immediate.

\begin{prop}
Let $A$ be an algebra. Then the exchange graph of $\tau$-tilting pairs of $A$ is a full subquiver of $\mathfrak{G}_A$. Moreover both quivers are isomorphic if $A$ is $\tau$-finite.
\end{prop}

\begin{proof}
It follows directly from Proposition \ref{chambers} and Theorem \ref{norefinements}.
\end{proof}

\begin{conj}
The quiver $\mathfrak{G}_A$ is isomorphic to the exchange graph of $\tau$-tilting pairs for every algebra $A$.
\end{conj}

\subsection{Maximal green sequences}

We first show in this subsection  that every maximal green sequence in the module category of an algebra is induced by a green path.

\begin{theorem}\label{carac}
Let $A$ be an algebra and $\gamma$ a green path in the wall and chamber structure of $A$. Then $\gamma$ induces a maximal green sequence if and only if $\gamma$ crosses only finitely many walls $\D(N_1), \D(N_2), \dots, \D(N_n)$ such that $\phi_\gamma(N_i)\neq \phi_\gamma(N_j)$ if $i\neq j$. Moreover every maximal green sequence is induced by a green path of this form.
\end{theorem}

\begin{proof}
Let $\gamma$ be a green path and, once again, fix $\alpha=(1,1,\dots, 1)$. Then the stability function $\phi_\gamma$ is discrete if and only if $\gamma$ crosses only finitely many walls $\D(N_1)$, $\D(N_2)$, $\dots$, $\D(N_n)$ such that $\phi_\gamma(N_i)\neq \phi_\gamma(N_j)$ whenever $i\neq j$. Hence Theorem \ref{maximalgreensequences} implies that $\gamma$ induces a maximal green sequence if and only if $\gamma$ is of this form. This completes the proof of the first part of the statement. 

Let 
$$\{0\}=\T_{0}\subsetneq \T_{1} \subsetneq \cdots \subsetneq \T_{n}=\mod A$$
be a maximal green sequence. The aim is to show that this maximal green sequence is induced by a green path $\gamma$.

As a first step, we note that by Corollary \ref{SVMmod} this sequence is of the form
$$\{0\}=\Fac M_0\subsetneq \Fac M_1\subsetneq \dots\subsetneq \Fac M_n=\mod A$$
for a given set of $\tau$-tilting pairs $\{(M_i,P_i)\}_{i=0}^n$ such that $(M_0,P_0)=(0,A)$, $(M_n,P_n)=(A,0)$ and $(M_i,P_i)=\mu(M_{i-1},P_{i-1})$. 

Let $(M_i,P_i)$ be one of the $\tau$-tilting pairs in the set associated to our maximal green sequence. Then is easy to see that for every $i$, there exists an $\varepsilon_i>0$ such that $(\varepsilon_i(-1,-1,\dots, -1)+\alpha(M_i,P_i))\in\Ch_{(M_i,P_i)}$ and set $\varepsilon=\min_{i=1}^{n-1}\{\varepsilon_i\}$.

Now we define for each  $1\leq i\leq n$ the function $\gamma_i$ on the interval $\left[\frac{n-i}{n},\frac{n-i+1}{n}\right]$ by
\begin{eqnarray*}
\gamma_i(t) & = & (n(1-t)-i+1)\left(\alpha(M_i,P_i)+\frac{(n-i)\varepsilon}{n}(-1,-1,\dots, -1)\right)+ \\
 & & (n(t-1)+i)\left(\alpha(M_{i-1},P_{i-1})+\frac{(n-i+1)\varepsilon}{n}(-1,-1,\dots, -1)\right),
 \end{eqnarray*}
 and an extra function
$$\gamma_0=(n(1-t))\left(\alpha(M_1,P_1)+\frac{(n-1)\varepsilon\alpha}{n}(M_0-P_0)\right)+(n(t-1)+1)\alpha(M_0,P_0)$$ defined on the interval $\left[\frac{n-1}{n},1\right]$. Then we define the function $\gamma$ on the interval $[0,1]$ as $\gamma(t)=\gamma_i$ whenever $t\in\left[\frac{n-i}{n},\frac{n-i+1}{n}\right]$. We claim that $\gamma$ is a green path. 

First, it follows from the definition of $\gamma$ that $\gamma(0)=(1,1,\dots, 1)$ and $\gamma(1)=(-1,-1,\dots, -1)$. 

Let $N$ be an arbitrary $A$-module and consider the function $\rho_N:[0,1]\to \mathbb{R}$ defined by $\rho_{N}(t)=\theta_{\gamma(t)}([N])$. Then we have that $\rho_N(0)>0$ and $\rho_N(1)<0$. Moreover 
$$\rho_N\left(\frac{i+1}{n}\right) \; \leq \; \rho_N\left(\frac{i}{n}\right)+\frac{\varepsilon}{n}\rho_N(1) \; < \; \rho_N\left(\frac{i}{n}\right)$$
by Lemma \ref{decrisingfunction}. Hence there is a unique $t_N\in[0,1]$ such that $\rho_N(t_N)=0$ because $\gamma$ is a piecewise linear function. Therefore $\gamma$ is a green path.

On the other hand, Remark \ref{sametorsion} implies that the torsion classes induced by $\phi_\gamma$ are $\T_{\frac{i}{n}}=\Fac M_i$ for every $i$. Therefore, the chain of torsion classes induced by the green path $\gamma$ is a refinement of the maximal green sequence
$$\{0\}=\T_{0}\subsetneq \T_{1} \subsetneq \cdots \subsetneq \T_{n}=\mod A.$$ 
But a maximal green sequence does not allow any refinements, thus we have constructed a green path $\gamma$ inducing the maximal green sequence $$\{0\}=\T_{0}\subsetneq \T_{1} \subsetneq \cdots \subsetneq \T_{n}=\mod A.$$

This finishes the proof.
\end{proof}

\subsection{An application}

We would like to illustrate our results by two examples. The first one shows three green paths in the wall and chamber structure displayed in Figure \ref{w&cA2}.

\begin{ex}
Let $A$ be the path algebra of the quiver $\xymatrix{1\ar[r]&2}$. In Figure \ref{troischemins} we describe three different green paths. One can see that $\gamma_1$ and $\gamma_3$ have the properties described in Theorem \ref{carac}. Therefore these two green paths induce the two maximal green sequences which exist in $\mod A$. On the other hand, $\gamma_2$ does not induce a maximal green sequence because this green path crosses every wall at the same time. 

\definecolor{qqzzqq}{rgb}{0.,0.6,0.}
\definecolor{ffqqqq}{rgb}{1.,0.,0.}
\definecolor{qqccqq}{rgb}{0.,0.8,0.}
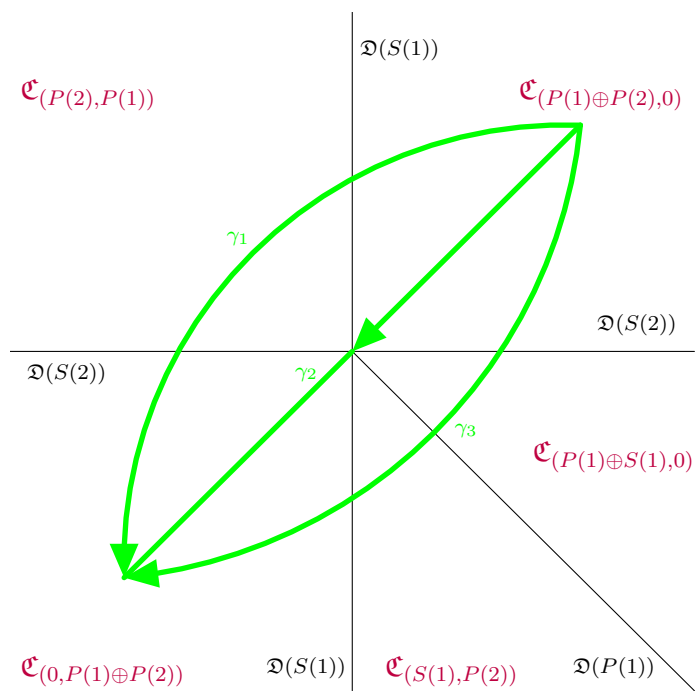
\begin{figure}
\begin{center}
				\begin{tikzpicture}[line cap=round,line join=round,>=triangle 45,x=3.0cm,y=3.0cm]
					\clip(-1.5,-1.5) rectangle (1.5,1.5);
						\draw (0.,0.) -- (0.,1.5);
						\draw (0.,0.) -- (0.,-1.5);
						\draw [domain=-1.5:0.0] plot(\x,{(-0.-0.*\x)/-1.});
						\draw [domain=0.0:1.5] plot(\x,{(-0.-1.*\x)/1.});
						\draw [domain=0.0:1.5] plot(\x,{(-0.-0.*\x)/1.});
						\draw[color=purple] (1.5,1.25) node[anchor=north east] {$\Ch_{(P(1)\oplus P(2),0)}$};
						\draw[color=purple] (-1.5,1.25) node[anchor=north west] {$\Ch_{(P(2),P(1))}$};
						\draw[color=purple] (1.55,-0.35) node[anchor=north east] {$\Ch_{(P(1)\oplus S(1), 0)}$};
						\draw[color=purple] (0.1,-1.3) node[anchor=north west] {$\Ch_{(S(1),P(2))}$};
						\draw[color=purple] (-1.5,-1.3) node[anchor=north west] {$\Ch_{(0,P(1)\oplus P(2))}$};
                        \draw [shift={(-1.214762126142883,1.214762126142883)},line width=2.pt,color=green,<-]  plot[domain=4.809055239892678:6.186519047671598,variable=\t]({1.*2.2251503423864096*cos(\t r)+0.*2.2251503423864096*sin(\t r)},{0.*2.2251503423864096*cos(\t r)+1.*2.2251503423864096*sin(\t r)});
						\draw [shift={(0.9296624106997075,-0.9296624106997076)},line width=2.pt,color=green,->]  plot[domain=1.5343617347889673:3.178027245595722,variable=\t]({1.*1.9309439131512813*cos(\t r)+0.*1.9309439131512813*sin(\t r)},{0.*1.9309439131512813*cos(\t r)+1.*1.9309439131512813*sin(\t r)});
						\draw [line width=2.pt,color=green] (-1.,-1.)-- (0,0);
						\draw [line width=2.pt,color=green,->] (1.,1.)-- (0,0);

					\begin{scriptsize}
                    	\draw[color=green] (-0.5,0.5) node {$\gamma_1$};
						\draw[color=green] (-0.2,-0.1) node {$\gamma_2$};
						\draw[color=green] (0.5,-0.35) node {$\gamma_3$};

						\draw[color=black] (0,1.25) node[anchor= south west] {$\mathfrak{D}(S(1))$};
						\draw[color=black] (-0.2,-1.4) node {$\mathfrak{D}(S(1))$};
						\draw[color=black] (-1.25,-0.1) node {$\mathfrak{D}(S(2))$};
						\draw[color=black] (1.15,-1.4) node {$\mathfrak{D}(P(1))$};
						\draw[color=black] (1.25,0.12) node {$\mathfrak{D}(S(2))$};
					\end{scriptsize}
				\end{tikzpicture}

\end{center}
\caption{Three green paths inducing different chains of torsion classes}
\label{troischemins}
\end{figure}
\end{ex}

In the second example, we apply  the results developed in this paper to show that the algebras related to the so-called \textit{Markov quiver} do not admit maximal green sequences. These algebras are related to the cluster algebra of the one-punctured torus, and  have been object of intense studies in the context of cluster algebras, see for instance \cite[Example 35]{L-F-surfaces&potentials}, \cite{N-C-c&gvect} or \cite[Theorem 5.17]{DIJ}.

\begin{theorem}
Let $A=kQ/I$ be an algebra where  $I$ is an admissible ideal of $kQ$ and the quiver $Q$ has exactly three vertices and admits the quiver
$$
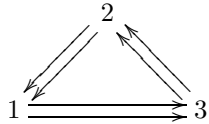
\xymatrix{
 & 2\ar@<0.5 ex>[dl]\ar@<-0.5 ex>[dl] & \\
 1\ar@<0.5 ex>[rr]\ar@<-0.5 ex>[rr] & & 3\ar@<0.5 ex>[ul]\ar@<-0.5 ex>[ul] }$$
as a subquiver. Then there is no maximal green sequence in $\mod A$.
\end{theorem}

\begin{proof}
We start the proof showing that if $C$ is the path algebra of the quiver 
$$
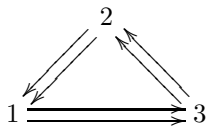
\xymatrix{
 & 2\ar@<0.5 ex>[dl]\ar@<-0.5 ex>[dl] & \\
 1\ar@<0.5 ex>[rr]\ar@<-0.5 ex>[rr] & & 3\ar@<0.5 ex>[ul]\ar@<-0.5 ex>[ul] }$$
modulo its radical squared, then there is no maximal green sequence in $\mod C$. 

To do so, we denote by $Q_1$, $Q_2$ and $Q_3$ the $C$-modules having $(1,0,1)$ $(1,1,0)$ and $(0,1,1)$ as a dimension vector, respectively. Note that these modules correspond to the quasi-simple modules living in the tubes of the Auslander-Reiten quiver of $\mod C$. 

Let $\gamma$ be a green path in the wall and chamber structure of $\mod C$. Then, by definition, $\gamma(0)=(1,1,1)$ which is a interior point of the chamber $\Ch_{(C,0)}$ defined by the $\tau$-tilting pair $(C,0)$. 
Moreover, the description of the chamber $\Ch_{(C,0)}$ induced by $(C,0)$ given in Corollary \ref{walls&chambers} implies that there exists a simple module $S$ such that $t_S\leq t_M$ for every $C$-module $M$. If $S=S(1)$ ($S=S(2)$, $S=S(3)$) then $Q_2$ ($Q_3$, $Q_1$, respectively) is $\phi_\gamma$-stable. In that case we have that $\T_{Q_2}$ ($\T_{Q_3}$, $\T_{Q_1}$respectively) is not a functorially finite torsion class. In each case, this implies that $\gamma$ is crossing an infinite number of walls in the wall and chamber structure of $C$ by Corollary \ref{SVMmod}. Hence $\gamma$ does not induce a maximal green sequence. Since Theorem \ref{carac} says that every maximal green sequence is given by a green path and there is no green path inducing a maximal green sequence in $\mod C$, this implies that there is no maximal green sequence in $\mod C$. 

Now, let $A$ be an algebra as stated in the theorem. Then there exists an ideal $I'$ contained in the radical of $A$ such that $C=A/I'$. Since every $C$-module is an $A$-module, we get that every wall in the wall and chamber structure of $C$ is a wall in the wall and chamber structure of $A$. Therefore each green path in the wall and chamber structure of $A$ crosses at least as many walls as it does in the wall and chamber structure of $C$, implying that there is no maximal green sequence in $\mod A$. This finishes the proof.
\end{proof}

\def\cprime{$'$} \def\cprime{$'$}

\end{document}